\documentclass[12pt]{article}

\usepackage{setspace}
\usepackage[table]{xcolor}
\usepackage{amsmath, amssymb, amsthm, graphicx,bm}
\usepackage{tikz}
\usetikzlibrary{tikzmark,bending,arrows.meta}
\usepackage[ total={6.5in, 8.5in}]{geometry}
\usepackage{mathtools}
\usepackage[normalem]{ulem}
\usepackage{multirow}
\usepackage{xcolor}
\usepackage[
backend=biber, 
style=nature,
sorting=nyt
]{biblatex}
\newcommand{\ruleset}[1]{\textsc{#1}}
\newcommand{\PN}[2]{\textnormal{PN}\ensuremath{\left(#1,#2\right)}}
\newcommand{\CN}[2]{\textnormal{CN}\ensuremath{\left(#1,#2\right)}}
\newcommand{\SN}[2]{\textnormal{SN}\ensuremath{\left(#1,#2\right)}}
\newcommand{\defeq}{\overset{\mathrm{def}}{=\joinrel=}}
\newcommand{\IRP}{Invariance Reduction Process}

\def\MS{\text{(MS\ensuremath{'})}}
\def\TS{\text{(TS\ensuremath{'})}}
\def\LT{\text{L(TS\ensuremath{'})}}
\def\RT{\text{R(TS\ensuremath{'})}}
\def\LM{\text{L(MS\ensuremath{'})}}
\def\RM{\text{R(MS\ensuremath{'})}}

\def\N{\ensuremath{\mathcal{N}}}

\def\P{\ensuremath{\mathcal{P}}}

\def\A{\ensuremath{\mathcal{A}}}
\def\B{\ensuremath{\mathcal{B}}}
\def\p{\ensuremath{\bm{p}}}

\def\z{\ensuremath{\bm{z}}}

\def\Nim{\ruleset{Nim}}

\newtheorem{theorem}{Theorem}[section]
\newtheorem{proposition}[theorem]{Proposition}

\newtheorem{lemma}[theorem]{Lemma}
\newtheorem{definition}[theorem]{Definition}
\newtheorem{remark}[theorem]{Remark}
\newtheorem{example}[theorem]{Example}

\newtheorem{conjecture}{Conjecture}
\newtheorem{open}{Open Problem}

\newcommand{\cref}[1]{Corollary~\textup{\ref{#1}}}

\definecolor{darkgreen}{rgb}{0.0, 0.8, 0.0}

\DeclareMathOperator{\mex}{mex}
\DeclareMathOperator{\im}{im}

\title{The \IRP~- a New Tool to Solve Circular Nim and Related Games}

\author{Balaji R. Kadam\footnote{Indian Institute of Technology Madras, Chennai, India}  
	\and 
	Matthieu Dufour\footnote{
    Universit\'e du Qu\'ebec \`a Montr\'eal, Montr\'eal, Canada 
    } 
    \and 
    Silvia Heubach\footnote{California State University Los Angeles, Los Angeles, USA
	}}

\date{\today}

\begin{document}
\maketitle

\begin{abstract}
We introduce the notion of invariant vectors of a game and develop the \IRP, which first uses reduction of positions via invariance and then zero and merge reductions of games to arrive at smaller, solved sub-games for closed subspaces of the positions. This process makes it much easier to prove that there are moves from \N-positions to \P-positions, and can also be used in some cases to show that there are no moves between \P-positions. This process is suitable for all variations of \ruleset{Nim}~whose rule sets form a simplicial complex. We rephrase \ruleset{SimplicialNim}  as \ruleset{SetNim} \SN{n}{\A} and derive results on the structure of the \P-positions in terms of invariant vectors, without needing the background and notation of simplicial complexes. We also show that invariant vectors differ from the circuits used to describe the \P-positions in  \ruleset{Simplicial Nim} and that invariant vectors have wider applicability compared to circuits. We apply the \IRP~to derive  results on the \P-positions of the family of \ruleset{PathNim} games where play is allowed on at least half the stacks, as well as for the \ruleset{CircularNim} games \CN{n}{k} with $n=7, k=3$ and $n=8,k=3$.
\vspace{.1in}

\end{abstract}

\section{Introduction}\label{sec:intro}

\Nim~is a well-known combinatorial game that is at the core of combinatorial game theory. It was named and solved by C. Bouton~\cite{Bo1901}. \Nim~is played on $n$ stacks of tokens, and in each move, a player selects one stack and removes at least one, and as many as all tokens, from the stack. Due to the central role of \Nim, many variations of \Nim~have been introduced over the years. The most obvious generalization is to allow play on up to $k$ stacks, a game introduced by Moore~\cite{Mo1910}, and now known as \ruleset{Moore's $k$-Nim}. A ``slow" version of \ruleset{Moore's $k$-Nim} was studied by~\cite{GHHC20}, where in each move, only a single token can be removed from the selected stacks. Other variations have imposed a geometric structure on the stacks, for example, \ruleset{CircularNim} \CN{n}{k}, where $n$ stacks of tokens are arranged in a circle. In each move, a player selects $k$ consecutive stacks and then removes at least one and as many as all tokens (in total) from the $k$ stacks. \ruleset{CircularNim} \CN{n}{k} has been solved for some families of games as well as individual games~\cite{DuHe09,DHV21}.  Other variations include playing \Nim~on a graph~\cite{BeCoGu2001}, where tokens are placed on the edges of the graph and then a player selects a vertex and removes at least one (and as many as all) token(s) in total from the adjacent edges. A variation on that theme is \ruleset{MatrixNim}~\cite{Hol1958}, which is played on a grid representing $n$ rows and $m$ columns, with stacks of tokens at the grid points. In each move, a player can play on the stacks of a single row, or on any of the stacks that are contained in $m-1$ columns, removing at least one token from one of the stacks.   A more recent paper~\cite{KaSh2025} considered a number of variations of playing \Nim~on the faces of a cube, where tokens are placed on the faces, and the different variations arise from whether a vertex, an edge, a face, or a rotation is selected by the player. In each move, at least one of the tokens on faces adjacent to the selected cube component are removed. We highlight these specific variations of \Nim~as they all can be described as \ruleset{SetNim} games, which we will define below. 

\Nim~and its variations are impartial combinatorial games, that is, two-player games with no hidden information or randomness where both players have the same moves available to them. The positions of impartial games fall into one of two categories, \N-positions (those from which the next player can win if playing optimally) and \P-positions (those from which the next player is bound to lose if the opponent plays optimally). Thus, knowing the \P-positions allows a player in a \N-position to win by always moving to a \P-position. The set of \P-positions is uniquely determined by the fact that there is no move from a \P-position to another \P-position, and that from every \N-position, there is a move to a \P-position. We will assume a basic familiarity of the reader with impartial combinatorial games. Details can be found in~\cite{AlNoWo2019, BeCoGu2001, Sie2013}.

\Nim~and the variations described above fall under the umbrella of \ruleset{SimplicialNim}, which was studied by Ehrenborg and Steingrimmson~\cite{ES1996} and by Horrocks~\cite{Hor2010}.  The authors use language from algebraic topology and algebra to describe a general \Nim~game played on the vertices of a graph, with the allowed moves described by sets of vertices. \Nim~and its variations discussed above have in common that once a player selects a set of vertices to play on, s/he can play on any or all of these stacks. This implies that if $A$ is a set of vertices that describes an allowed move, then any set $B \subseteq A$ also represents an allowable move set. Thus the allowable move sets form a simplicial complex, as play on any individual stack is always allowed (see Definition \ref{def:simpcomp}).

We will now give a slightly modified definition of the \ruleset{SimplicialNim} game, which we will refer to as \ruleset{SetNim}. This modified definition is, in our opinion, more natural for a reader without knowledge of the extensive definitions and notation of simplicial complexes. In addition, our definition of the game will allow us to easily describe reductions of games to sub-games without cumbersome notation. However, since our definition is just a restatement, we will still be able to utilize certain results and definitions (appropriately restated) from~\cite{ES1996}. In addition, our focus will be on invariant vectors, which play a role similar to circuits in simplicial complexes, but differ from them for most games. 

\begin{definition}
 The game \ruleset{SetNim} \SN{n}{\A} is played on a set of $n$ stacks of tokens which are placed on the vertices $V$ of a graph.\footnote{Vertices may be labeled as $1,\dots,n$, as $0,1,\dots, n-1$, or as  $a,b,c,\dots$, depending on what is convenient to describe the move set. The labeling usually mirrors the way we describe the positions of the game.} The collection of sets $\A=\{A_1,A_2,\ldots,A_{\ell}\}$ with $\cup_{i=1}^{\ell} A_i =V$ defines the allowed moves, where each set $A_i$ in the collection represents a set of vertices that the player can play on. Playing on such a set means to take at least one 
token from at least one stack of the particular set. We will only list the maximal allowable move sets. The last player to move wins. 
\end{definition}

The unique terminal position in \SN{n}{\A} is the position where all stacks are zero, which we will denote by $\bm{0}$. We start by giving some examples of known games that can be described as a \ruleset{SetNim}~game. 

\begin{example}
    The game of \Nim~on $n$ stacks can be expressed as the game \SN{n}{\A_1}~where $\A_1$ is the set of singletons, $\{\{1\},\ldots,\{n\}\}$. \ruleset{Moore's $k$-Nim} can be described as \SN{n}{\A_2}~where $\A_2$ is the collection of all $k$-element subsets of $\{1, \ldots,n\}$.    \ruleset{CircularNim} \CN{n}{k} is played on  $k$ consecutive stacks of a cycle graph. Thus, \CN{n}{k}~equals \SN{n}{\A_3}~with $$\A_3=\{\{i,i+1,\ldots, i+k-1\} \mid  i=0,\ldots, n-1\},$$ where the vertex labels in $\A_3$ are understood to be modulo $n$.
\end{example}

Another game that can be described under the umbrella of \ruleset{SetNim}~is the game \ruleset{PathNim}. 

\begin{definition} The game \ruleset{PathNim} \PN{n}{k} is played on stacks of tokens arranged on a path with $n$ vertices.\footnote{This game was described in~\cite{ES1996} with a different meaning of  parameter $n$ but identical meaning of parameter $k$.}  In each move, a player selects $k$ consecutive vertices and removes at least one token  from at least one of the selected stacks. The last player to move wins. This corresponds to \ruleset{SetNim}~\SN{n}{\A}~with 
$$\A=\{\{i,i+1,\ldots, i+k-1\} \mid  i=1,\ldots, n-k+1\}. $$
\end{definition}

Note that a position of \ruleset{CircularNim} \CN{n}{k} that contains a segment of at least $k-1$ consecutive zeros naturally reduces to \ruleset{PathNim} \PN{n-(k-1)}{k}. This indicates that knowledge of how to win certain \ruleset{PathNim} games will be important in solving \ruleset{CircularNim} games. In Theorem~\ref{thm:P-pos path} we will give results on the \P-positions of \ruleset{PathNim} games where play is allowed on at least half of the stacks. \\

\noindent Here is an overview of the contents of the sections that follow and our main results: 
\begin{itemize}
  \item In Section~\ref{sec:background} we will define our basic notation as well as zero and merge reduction of games to sub-games. The resulting sub-games are congruent to the game from which they originate in terms of classification of associated positions as \N- or \P-positions and correspondence of the moves between pairs of associated positions. This provides one key ingredient in our reduction process. We provide examples of these reductions and a summary of the \P-positions of known games that we use in subsequent sections. 
    \item In Section~\ref{sec:invar} we define invariant vectors for a set of positions and outline their usefulness in both finding the pattern of the \P-positions of a game and in proving the hypothesized pattern. A key tool is the {\bf \IRP}, which simplifies the proof that from any \N-position of the game there is a move to a \P-position. We prove that the \IRP ~produces such a move for each subspace of \P-positions. 
       \item Many of the reductions that occur when using the \IRP~lead to \ruleset{PathNim} games. In Section~\ref{sec:pathnim} we provide a result on the {\bf \P-positions of PN}$\bm{(n,k)}$ {\bf for }$\bm{k \geq \lceil \frac{n}{2}\rceil}$, that is, when play is allowed on at least half of the stacks  (Theorem~\ref{thm:P-pos path}). We also {\bf solve the game SN}$\bm{(6,\A)}$ with $\A=\{\{a,b,c\},\{b,c,d\},\{c,d,e\}$, $\{d,e,f\},\{a,f\}\}$ (Theorem~\ref{thm:H}). This game plays an crucial role when solving \CN{7}{3}.  
    \item In the following two sections we use the \IRP~to {\bf solve two \ruleset{CircularNim} games}: {\bf CN(7,3)} 
    in Section~\ref{sec:CN_7-3} ( Theorem~\ref{thm:P7-3}) and {\bf CN(8,3)} 
    in Section~\ref{sec:CN_8-3} (Theorem~\ref{thm:P83}).
    \item Section \ref{sec:Simp} takes us on  a little detour into simplicial complexes.  We provide the necessary definitions that allow us to {\bf compare and contrast invariant vectors and the circuits of simplicial complexes} and their role in determining the structure of the \P-positions of a \ruleset{SetNim} game. 
    \item In Section~\ref{sec:future}, we conclude with a summary and pose  open questions and conjectures for future study.
\end{itemize}

We start by defining our notation and important terms.

\section{Notation and Game Reductions} \label{sec:background}

A \emph{position} of a \ruleset{SetNim} game is a vector of the stack heights, $\p=(p_1,\ldots, p_n)$. We will write $\p \rightarrow \p'$ for a legal (allowed) move from $\p$ to an \emph{option} $\p'=(p_1',\ldots, p_n')$ and denote the minimum stack height of a position $\p$ by $\min(\p)$. The set of \P-positions of a game $G$ is denoted by $\P_G$. For ease of notation, we will use  $\P_{n,k}$ for either  \CN{n}{k} or \PN{n}{k}  games whenever the type of game is clear from the context. We will also refer to a stack or its number of tokens interchangeably when it is clear from the context what we mean.

The techniques employed in the proofs of our results can be described broadly as using reductions of both positions and games to utilize results on the \P-positions of smaller games to solve the larger games at hand. 

\begin{definition} Let $G=$ \SN{n}{\A} be a \ruleset{SetNim}~game on a graph with vertex set $V$. A game $\tilde{G}=$ \SN{m}{\B} on a subset of vertices $\tilde{V} \subset V$ is a \emph{sub-game} of $G$ if $m < n$ and for any $A \in \A$, there is a set $B\in \B$ such that the restriction $A\cap \tilde{V}$ of $A$ to $\tilde{V}$ satisfies  $A\cap \tilde{V} \subseteq B$.
\end{definition}

There are two main ways to reduce games with a larger number of stacks to games with a smaller number of stacks,  the \emph{zero reduction} and the  \emph{merge reduction}. A zero reduction can be used when, during game play, one or more stacks have been completely depleted. Such stacks no longer impact the game as one cannot play on them.  We remove these stacks from the sets $A_i$ because a move on a set of stacks that include an empty stack is equivalent to the corresponding move on only the non-empty stacks. Here is the formal definition. 

\begin{definition} [Zero reduction] Let $G=$ \SN{n}{\A} be a \ruleset{SetNim}~game on an underlying graph with vertex set $V.$ Let $Z=\{i \mid p_i=0\}$. Then the zero reduction sub-game $\tilde{G}$ is played on the stacks $\tilde{V}=V-Z$, and $\tilde{A_i}={A_i} \cap \tilde{V}$.
\end{definition}

\begin{remark} \label{rem:zero red} There is a one-to-one correspondence between $\p\in G$ and $\tilde{\p}\in \tilde{G}$ because $\tilde{p}_i=p_i$ for $i \notin Z$ and $p_i=0$ for $i\in Z$. Thus,  the two games have isomorphic game trees and therefore, their sets of \P-positions are identical with the adjustment made for the zero stacks of $\p$. 
\end{remark}

A merge reduction arises when there is a subset of vertices $C$ that have the property that if a move is allowed on one of those stacks, then one can also play on all of the other stacks of $C$. In this case, we can treat this collection of stacks as if it was a single stack for which the number of tokens is the sum of the tokens on the vertices of $C$. Replacing the individual stacks in the sets $A_i$ with the newly combined stack reduces the number of total stacks in the game, and potentially the maximal number of stacks we are allowed to play on.  Here is the formal definition.

\begin{definition} [Merge reduction]\!\!\footnote{This definition is a restatement of Definition 6.6 of~\cite{ES1996}.} Let $G=$ \SN{n}{\A} be a \ruleset{SetNim}~game on an underlying graph with vertex set $V$. Assume that a collection of stacks $C\subset V$ has the following property: For any set $A_i \in \A$, either $C\subseteq A_i$ or $C \cap A_i = \emptyset $. Then the merge reduction sub-game $\tilde{G}_{v^*}^C$ is played on the vertices $\tilde{V}=(V-C)\cup v^*$ (vertices of $C$ have been merged into a new vertex $v^*$) and the move sets  $\tilde{A}_i$ are defined by
$$\tilde{A_i}=\left\{\begin{tabular}{ll}
    $A_i$ & \text{if } $A_i\cap C=\emptyset$ \\
  $( A_i-C)\cup v^* $ & \text{otherwise }
\end{tabular}\right.$$
for $A_i \in \A$.
If $\p$ is a position of $G$, then the associated position $\tilde{\p}$ of $\tilde{G}_{v^*}^C$ is defined as $\tilde{p}_{v^*}=\sum_{i\in C}p_i$, and $\tilde{p}_i=p_i$ for  $i \notin C$.
 \end{definition}

Note that $\p$ and $\tilde{\p}$ have the same overall number of tokens. Furthermore, while there is a single reduced position $\tilde{\p}$ for every $\p\in G$, there are multiple positions $\p\in G$ that have the same reduced position. We will call the games $G$ and $\tilde{G}_{v^*}^C$ \emph{congruent} as their game outcomes are the same, with the merge-reduced game playing the role of a canonical form. Before we can state the relevant lemma, we remind the reader of the Sprague-Grundy function $g(\p)$, which is defined recursively as $g(\p)=\mex(\{g(\p') \mid \p \rightarrow \p'\})$, where $\mex(S)=\min(\mathbb{N}_0-S)$, that is, the smallest non-negative integer missing from the set $S$. The set of positions with Sprague-Grundy value zero is precisely the set of \P-positions of a game. 
\begin{lemma} \label{lem:merge red}
 Let $\p$ be a position of a game $G$ and $\tilde{\p}$ the associated position in the merge-reduced game $\tilde{G}^C_{v^*}$. Then the Grundy values $g(\p)$  and $g(\tilde{\p})$ are the same. In particular, $\p$ is a \P-position of $G$ if and only if  $\tilde{\p}$ is a \P-position of $\tilde{G}^C_{v^*}$. \end{lemma}

 Note that Lemma~\ref{lem:merge red} corresponds to Corollary 6.10 of~\cite{ES1996}, for which the proof was left to the reader. Therefore, we will give the proof here. 

 \begin{proof} We proceed by induction on $s = \sum_i p_i=\sum_j\tilde{p}_j$. The base case $s=0$ implies that both $\p$ and $\tilde{\p}$ equal the terminal position $\bm{0}$ in their respective games,  thus $g(\p)=g(\tilde{\p})=0$. Now assume that $s=n$ and the induction hypothesis holds for $s<n$. Then 
 \begin{align*}
     g(\p)=\mex(\{g(\p') \mid \p \rightarrow \p'\})=\mex(\{g(\tilde{\p}') \mid \tilde{\p} \rightarrow \tilde{\p}'\})=g(\tilde{\p})
 \end{align*}
 because for every move from $\p$ to $\p'$ there is a corresponding (legal) move from $\tilde{\p}$ to $\tilde{\p}'$, the associated reduced positions, and the Grundy values of the options are the same by induction hypothesis since at least one token was removed in the move.  Note that the options for $\p$ might be more numerous, but all of them have a corresponding option of $\tilde{\p}$ to which they map. Thus, the set of Grundy values considered in the mex functions are identical. 
 \end{proof}

We now give examples of the two types of reductions.  Note that a zero reduction may in turn lead to a merge reduction. For ease of readability, we will use the notation $G\defeq \A$ as a shorthand for ``$G$ is the game \SN{n}{\A} with allowed moves $\A$'' and $G\cong \tilde{G}$ if the games $G$ and $\tilde{G}$ are isomorphic, that is, they have the same move sets. Zero reduction will play an important role in considering subspaces of positions that have a specific number and relative position of zeros.

\begin{example} \label{ex:merge red} The following games illustrate merge and zero reductions.
\begin{enumerate}    
     \item \label{it:CN62sub} The game $G_1  \defeq \{ \{a,b \},\{b,c \},\{c,d \},\{d,e \},\{e,f \}, \{f,a \}\}\cong $ \CN{6}{2}, when played on the subspace of positions with  $d=f=0$, reduces (via zero reduction) to the game $\tilde{G}_1 \defeq \{ \{a,b \},\{b,c \},\{e\}\}\cong \{ \{a,b \},\{b,c \},\{d\}\}$. No further reduction is possible.
    \item Let $G_2 \defeq \{ \{a,d \},\{a,b,c \},\{b,c,d \}\} $. Notice that in each move set, the stacks $b$ and $c$ either both appear or are both absent,  so we can perform a merge reduction to obtain  $\tilde{G}^{\{b,c\}}_x \defeq \{ \{a,d \},\{a,x \},\{x,d \}\} \cong$  \CN{3}{2}. 
    \item Let $G_3 \defeq \{ \{a,b,c\},\{b,c,d\},\{c,d,e\},\{d,e,f\},\{e,f,a\},\{f,a,b\}\} \cong$ \CN{6}{3}. The zero-reduction when $a=b=0$ leads to
    $\tilde{G}_3 \defeq \{ \{c,d,e\},\{d,e,f\}\} \cong $ \PN{4}{3}. We can now perform a merge reduction to obtain the game $\tilde{G}^{\{d,e\}}_x=\{\{c,x\},\{x,f\}\}\cong $ \PN{3}{2}. 
\end{enumerate}
 \end{example}

In the proofs of the main results, we will use the \IRP~(see Section~\ref{sec:invar}) to find moves from \N-positions to \P-positions by reducing to a suitable sub-game and then identifying a winning move in the sub-game. As part of the process, we need to find the corresponding winning move in the original game. Example~\ref{ex:step3} shows how the mapping of the move $\hat{\bm{m}}$ in the merged game to a move $\bm{m}$ in the original game can be achieved for a particular move and merge position, using the game $G_2$  of Example~\ref{ex:merge red}. We will see the full \IRP~in action in the proof of Theorem~\ref{thm:H}.

\begin{example} \label{ex:step3}  Let  $G_2\defeq \{ \{a,d \},\{a,b,c \},\{b,c,d \}\}$ and $\p=(2, 3, 5, 4)$. This game can be reduced to $\tilde{G}^{\{b,c\}}_x \defeq \{ \{a,d \},\{a,x \},\{x,d \}\} \cong$  \CN{3}{2}. The associated reduced position is $\hat{\p}=(p_1,p_2+p_3,p_4)=(2,8,4)\notin \P_{3,2}$ (see Theorem~\ref{thm:sumCN}).  The winning move in \CN{3}{2} is to a position in which all stack heights are equal, that is, to $\hat{\p}'=(2,2,2)$ using the move $\hat{\bm{m}}=(0,6,2)$. Therefore, $m_{x}=6$, to be removed  across stacks $p_2=3$ and $p_3=5$. 
We can remove either one, two or three tokens from stack $p_2$ and then the remaining tokens from stack $p_3$, resulting in moves  $\bm{m}_1=(0,1,5,2)$, $\bm{m}_2=(0,2,4,2)$, and $\bm{m}_3=(0,3,3,2)$ in game $G_2$. In either case, the move is to $\p'$   with $p_1'=p_2'+p_3'=p_4'$. Note that because $G_2$ is equivalent to \CN{3}{2}, we can immediately describe the set of \P-positions as $\P_G=\{\p \mid p_1=p_2+p_3=p_4\}$.
\end{example}

Even more useful are reductions that apply to families of games. We give several examples of zero reductions, merge reductions, and the combination of the two.

\begin{example} \label{ex:red_fam} The following illustrates merge and zero reductions for families of games. 
\begin{enumerate}
 \item The game \CN{n}{k}, when played on the subspace of positions with $l \geq k-1$ adjacent zero stacks is equivalent, after applying zero reduction, to the game \PN{n-l}{k}.
  
 \item For any $n \geq 4$, we have that \PN{n}{n-1} $\cong$ \PN{3}{2}, because  \PN{n}{n-1}~$\defeq\{\{1,\ldots, n-1\},\{2,\ldots, n\}\}$, so we can merge-reduce  the set $C=\{2,\ldots,n-1 \}$ to a new stack $i^*$, resulting in the game $G^C_{i^*}\defeq\{\{1,i^*\},\{i^*, n\}\} \cong$ \PN{3}{2}.
 
 \item  The game  \CN{2\ell+1}{\ell+1}, when played on the subspace of positions with  $\ell-1$ consecutive zeros, that is, $\p=(x_1,0,\ldots,0,x_{\ell+1},x_{\ell+2},\ldots,x_{2\ell+1})$,   
 reduces (via zero reduction) to the game $\tilde{G}\defeq \{\{1,\ell+1\},\{\ell+1,\ell+2,\ldots,2\ell+1\},\{\ell+2,\ldots,2\ell+1,1\}\}$. Applying merge reduction to $\tilde{G}$ by combining vertices $\ell+2,\ldots,2\ell+1$  gives the game $\hat{G} \defeq \{\{1,\ell+1\},\{\ell+1,i^*\},\{i^*,1\}\} \cong$ \CN{3}{2}. Thus the \P-positions for this subspace are those for which $x_1=x_{\ell+1}=\sum_{i=1}^{\ell} x_{(\ell+1)+i}$ and $x_2=\cdots = x_{\ell}=0.$ (This result corresponds to Lemma 5 in~\cite{DHV21}, where it was proved in a slightly different manner.)

 \item \label{it:CNsubgame}The game \CN{c\, k}{k} is a sub-game  of \CN{c(k+1)}{k+1}. Specifically, the game \CN{c(k+1)}{k+1} played on the subspace of positions that have $c$ equally spaced zeros at distance $k+1$ apart can be zero-reduced to a game with $c \cdot k $ stacks, where play is on $k$ stacks. The reduction takes away $c$ zero stacks, reducing the total number of stacks to $c\cdot k$. Since each move set of $G$ contains exactly one zero, the maximal number of stacks one can play on is reduced by one to $k$.
 \end{enumerate}
\end{example}

In the proofs of our main results, we divide the space of \P-positions into subspaces that correspond to specific sub-games. For easy reference, we collect results on solved \ruleset{CircularNim} games (see~\cite{DuHe09,DHV21}) that will be used in examples and proofs. Note that for circular Nim, positions are only determined up to rotation and reflection. 

\begin{theorem}\label{thm:sumCN} 
    Let $\P_{n,k}$ denote the $\P$-positions of the game \CN{n}{k}. Then we have the following results:
    \begin{enumerate}
        \item $\P_{3,2}=\{(a,a,a)\mid a \ge 0\}$.
        \item $\P_{4,2}=\{(a,b,a,b)\mid a, b \ge 0\}$.
        \item $\P_{5,2}=\{(M, m, c, d, m) \mid M + m = c + d \text{ and } M = \max(\p)\}$.
        \item $\P_{5,3}=\{(0,b,c,d,b)\mid b=c+d\}$.
        \item $\P_{6,3}=\{(a,b,c,d,e,f)\mid a+b = d+e \text{ and } b+c = e+f\}$.
        \item $\P_{7,4}=S_1 \cup S_{2} \cup S_3 \cup S_{4}$, where
\begin{itemize}
    \item[-] $S_1 = \{\p \mid a = b = 0, \> c = g>0,\> d + e + f = c\}.$
    \item[-] $S_2 = \{\p \mid \p=(a,a,a,a,a,a,a)\}.$
    \item[-] $S_3 = \{\p \mid a = b, \>c = g, \>d = f,  \> a + c = d + e, \> \> 0< a < e\},$ and
    \item[-] $S_4 = \{\p \mid a = f, \> b + c = d + e = g + a, \>a < \min\{b,\>e\},a < \max\{c,d\}\}.$
\end{itemize}
    \end{enumerate}
\end{theorem}

\vspace{0.2in}

If the process of reduction leads to a sub-game whose solution is not known, then we cannot solve the larger game without first solving the sub-game. Specifically, Example~\ref{ex:red_fam}(\ref{it:CNsubgame}) with $c=3$ and $k=2$ implies that solving \CN{9}{3} is not possible without making progress on the unsolved game \CN{6}{2} first. By Example~\ref{ex:merge red}(\ref{it:CN62sub}), solving \CN{6}{2} in turn requires (among other things) solution of the game $G \defeq \{ \{a,b \},\{b,c \},\{d\}\}$, which itself is an open problem.

\section {Invariance and \P-positions} \label{sec:invar}

Invariant vectors will play a crucial role in solving the games in the subsequent sections. They are useful in finding and making conjectures about the patterns of the \P-positions,  for proving those patterns, and for efficient game play.  Let's get started with the definition.

\begin{definition} \label{def:inv} A vector is called a \emph{zero-one} vector if all its entries are either zeros or ones. A zero-one vector ${\bm z}$ is called \emph{invariant for a set $S$} of game positions if for all $\p \in S$ and any integer $c$ for which $\p+c\cdot \bm{z} \ge \bm{0}$, $\p\in S$ if and only if  $\p+c\cdot \bm{z} \in S$. 
\end{definition}

Because $\bm{0} \in \P$, the definition of invariant vectors immediately implies that the \P-positions of a game must contain any linear combination of the invariant vectors that result in non-negative stack heights.\footnote{When we write linear combinations of invariant vectors, we always assume that the coefficients are non-negative.} This ties the invariant vectors intimately to the \P-positions. How can we use this connection?

\subsection{Finding the \P-position patterns using invariant vectors} \label{subsec:pattern}

Finding the pattern of the \P-positions of a game often involves creating a reasonably large list of positions that have been classified as either \N- or \P-positions, and then trying to guess a pattern for the \P-positions. If there are groups of \P-positions that differ by multiples of a vector, then we have found a candidate for an invariant vector. To confirm the invariance, we check whether adding this vector to any of the positions (and their equivalent versions, such as rotations and reflections in \ruleset{CircularNim}) results in a position of the same class. If so, then the candidate vectors is indeed an invariant vector for the list of positions at hand. After checking any such candidate vectors, we can then check whether there is an overall pattern in the linear combinations of the invariant vectors. Here is an example of how invariant vectors can produce a conjecture for \P-positions. 

\begin{example} \label{ex:CN63lc} Assume that when analyzing the game \CN{6}{3}, you have identified the following as invariant vectors: $\bm{z}_1=(1,0,0,1,0,0)$, $\bm{z}_2=(0,1,0,0,1,0)$, $\bm{z}_3=(0,0,1,0,0,1)$,  $\bm{z}_4=(1,0,1,0,1,0)$, and $\bm{z}_5=(0,1,0,1,0,1)$. The set of \P-positions must contain the linear combinations of these. To highlight the structure that emerges, we will select convenient coefficient names, and  ``grow" the linear combinations step-by-step. First we add multiples of $\z_4$ and $\z_5$ to obtain $\p=x\cdot \z_4+y\cdot \z_5=(x,y,x,y,x,y).$
Then we add multiples of the three other vectors to obtain
$$\p=(x,y,x,y,x,y)+c_1\z_1+c_2\z_2+c_3\z_3=(x+c_1,y+c_2,x+c_3,y+c_1,x+c_2,y+c_3).$$ 
If we label the stacks in $\p$ as $(a,b,c,d,e,f)$, then the above pattern can equivalently be expressed as $\{\p \mid a+b=d+e,\, b+c=e+f,\,  c+d=f+a\}$, because each pair of stacks has an $x$ and a $y$ component, and stacks that are diagonally across have the same value $c_i$ added on. For \CN{6}{3}, the linear combinations of the invariant vectors actually equal the set of \P-positions, see Theorem~\ref{thm:sumCN}.
\end{example}

Before we tackle the question under what conditions the linear combinations of the invariant vectors equal \P, we will look at how properties of \P~impact the (non-)existence of invariant vectors and vice versa. 

\begin{remark} \label{rem:inv-spec} \hfill
    \begin{enumerate}
           \item In the \ruleset{PathNim} games, the only invariant vector is $\z=(1,0,\ldots,0,1)$ due to the string of consecutive zeros of the \P-positions. For the still-unsolved cases of \ruleset{PathNim} (those where play is on fewer than half the stacks), we know that the invariant vectors need to be symmetric with regard to reflection due to the underlying structure of \ruleset{PathNim}.
        
        \item In any \ruleset{CircularNim} game where the \P-positions have a designated zero, such as in the game \CN{5}{3}, there is no invariant vector. If there were a non-zero vector $\z$, then a $1$ entry in $\z$  would destroy the designated zero in one of the rotations of $\p+\z$.

        \item 
        In any \ruleset{CircularNim} game where the \P-positions are described via a structural relationship between locations of maxima or minima and the other stacks, the only potential invariant vector is $\z=(1,\dots,1)=\bm{1}$. For example, the \P-positions of \CN{5}{2} are given by $\P = \{(M, m, c, d, m) \mid M + m = c + d \text{ and } M = \max(\p)\}$. If the vector $\z$ would have a coordinate that is $0$, then we could find a position $\p \in P$ such that $\p+\z \notin \P$ for at least one of the rotations or reflections of $\p$. 
        
        \item If $\z=\bm{1}$, then any position that has equal stack heights is a \P-position. Furthermore, any conditions involving equality of sums of stacks must have the same number of stacks on both sides of the equation.  
        
        \item The converse, namely that  $(n,\ldots,n) \in \P_G$ for all $n \in \mathbb{N}_0$ implies that $\bm{1}$ is an invariant vector, is not true. A counter-example is the game \CN{7}{4}, where $\p=(0,0,1,0,0,1,1)\in S_1 \subset \P_{7,4}$, while $\p+\bm{1}=(1,1,2,1,1,2,2) \notin \P_{7,4}$ because it does not satisfy the conditions of any of the sets $S_i$ (see Theorem~\ref{thm:sumCN}). 

        \item \CN{n}{k} with $n \geq 3$ and $k \leq n/3$ has \Nim~on three stacks as a sub-game, because in any position with exactly three non-zero stacks at maximal distance from each other, play can be on exactly one of the non-zero stacks. Since the \P-positions of \Nim~on three stacks are defined by the Nim-sum,  no invariant vectors can exist for these games.
    \end{enumerate}
\end{remark}

While many questions regarding the interplay between \P-positions and the linear combinations of the invariant vectors remain to be explored, here is a first result on conditions for which the \P-positions consist exactly of the linear combinations of the invariant vectors. Example~\ref{ex:inv=P} illustrates the use of Proposition~\ref{prop:invPpos}. 

\begin{proposition} \label{prop:invPpos} Let $\p$ be a position of a \ruleset{SetNim} game $G$ with move set $\A$, $I$ be the set of invariant vectors $\z_i$ of the \P-positions of the game, and $\tilde{\z}=\sum_{i\in I} c_i\z_i$ where $c_i \geq 0$.  If for any $\p \in G$, we can find $\tilde{\p}=\p-\tilde{\z}$  such that the non-zero stacks of $\tilde{\p}$ are contained in some set $A \in \A$, then the set of \P-positions equals the set of linear combinations of the invariant vectors. 
\end{proposition}

\begin{proof} From the definition of invariance, we obtain that any vector $\tilde{\z}$ belongs to \P~since $\bm{0}\in\P$. Now assume that  $\p $ is not a linear combination of the invariant vectors. Then $\p-\tilde{\z} \neq \bm{0}$ for any $\tilde{\z}$. By the assumption on the non-zero stacks of $\tilde{\p}$ there is a move $\bm{m}$ from $\tilde{\p}=\p-\tilde{\z}$ to $\bm{0}$, so  $\tilde{\p} \in \N$, which implies that  $\p \in \N$ by the definition of invariance. Thus $\p \in \P$ if and only if $\p$ is a linear combination of the invariant vectors.
\end{proof}

 \begin{example} \label{ex:inv=P} Let $\p$ be a position in game \CN{6}{3}. Using the invariant vectors identified in Example~\ref{ex:CN63lc},  namely $\bm{z}_1=(1,0,0,1,0,0)$, $\bm{z}_2=(0,1,0,0,1,0)$,  $\bm{z}_3=(0,0,1,0,0,1)$,   $\bm{z}_4=(1,0,1,0,1,0)$, and $\bm{z}_5=(0,1,0,1,0,1)$, we show that we can create three consecutive zeros in the reduced position. 
 Let $c_i=\min(p_i,p_{i+3})$ for $i=1,2,3$, that is, $c_i$ is the minimum of the stack heights of diagonally opposite stacks. Then 
$\hat{\p} = \p - c_1 \bm{z}_1-c_2 \bm{z}_2-c_3\bm{z}_3$ has three zeros.  If, in position $\p$, two of the minima of pairs of diagonally opposite stacks are adjacent, then all three minima have to be adjacent, hence $\hat{\p}$ has three adjacent zeros. Otherwise,  $\hat{\p}$ has alternating zeros. We can create at least one additional zero by using the remaining two invariant vectors to reduce to $\tilde{\p}$ where 
 $$\tilde{\p} =\hat{\p}-\min(\hat{p}_1,\hat{p}_3,\hat{p}_5)\cdot \bm{z}_4-\min(\hat{p}_2,\hat{p}_4,\hat{p}_6)\cdot \bm{z}_5,$$
 which now has three consecutive zeros as claimed. As one can always play on the three non-zero stacks, Proposition~\ref{prop:invPpos} applies, and $\P_{6,3}=\{\p \mid \p=\sum_{i=1}^5 c_i \z_i\}$.
   \end{example}

\subsection{Proving the \P-position pattern using invariant vectors}

Invariant vectors can be used to greatly simplify the usually more difficult part of the proof, namely that there is a move from an \N-position to a \P-position, and can also be used to show that for a subset of \P-positions, there is no move to another \P-position. The usefulness of invariant vectors in the proofs of the pattern of the \P-positions stems from the following lemma. 

\begin{lemma}[Invariance Move]  \label{lem:inv}  Let $\p \notin \P$ and let $I=\{{\bm z}_1,\ldots, {\bm z}_l\} $ be a set of invariant zero-one vectors of \P~with $\tilde{\bm{z}}=\sum_{i=1}^l c_i\cdot \bm{z}_i$, where $c_i \geq 0$.  To find a move from  $\p$ to a position $\p' \in \P$ it suffices to find a (legal) move $\bm{m}$ from $\tilde{\p}=\p-\tilde{\bm{z}}\geq \bm{0}$ to a position $\tilde{\p}' \in \P$. Then $\p'=\tilde{\p}'+\tilde{\bm{z}}$ is in $\P$, and we can reach $\p'$ from $\p$ using the  move $\bm{m}$.
\end{lemma}

\begin{proof} Note that by iteratively evoking invariance of the $\z_i$, we obtain that if $\p \notin \P$, then $\tilde{\p}=\p-\tilde{\z} \notin \P$ also. Now assume we can find a legal move $\bm{m}$ from $\tilde{\p}$ to $\tilde{\p}' \in \P$. By once again iteratively invoking the invariance of the $\z_i$, we know that $\p':=\tilde{\p}'+ \tilde{\z}\in \P$. Since $\p'=(\p-\tilde{\z})-\bm{m}+\tilde{\z}=\p-\bm{m}$, we have found a legal move, namely $\bm{m}$, from $\p\notin \P$ to $\p' \in \P$. Figure~\ref{fig:inv} visualizes this process.
    \end{proof}

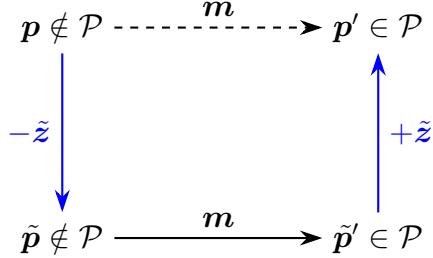
\begin{figure}[h!]
    \centering    
\begin{tikzpicture}[scale=0.7]

\node (A) at (-3,2) {$\p\notin \P$};
\node (B) at (3,2) {$\p' \in \P$};
\node (C) at (-3,-2) {$\tilde{\p} \notin \P$};
\node (D) at (3,-2) {$\tilde{\p}' \in \P$};

\draw[-{Stealth[length=3mm,width=2mm]}, thick, dashed] (A) -- (B) node[midway,above] {$\bm{m}$};
\draw[-{Stealth[length=3mm,width=2mm]}, thick] (C) -- (D) node[midway,above] {$\bm{m}$};

\draw[-{Stealth[length=3mm,width=2mm]}, thick, blue] (A) -- (C) node[midway,left] {$-\tilde{\z}$};
\draw[-{Stealth[length=3mm,width=2mm]}, thick, blue] (D) -- (B) node[midway,right] {$+\tilde{\z}$};

\end{tikzpicture}
    \caption{Using invariant vectors to find an option $\p' \in \P$ of $\p \notin \P$ that can be reached via legal move $\bm{m}$.}
    \label{fig:inv}
\end{figure}

\begin{example} \label{ex:inv} We illustrate the application of the above lemma using the game \CN{5}{2} and the position \(\p = (3,8,5,9,6) \notin \P_{5,2}.\) By Remark~\ref{rem:inv-spec}(3),  $\z = \bm{1}$ is an invariant vector for $\P_{5,2}$. Therefore, \(\tilde{\p} = \p-3\z=(0,5,2,6,3) \notin \P_{5,2}.\) One possible move to a \P-position is $\bm{m} = (0,0,2,4,0)$, which yields \(\tilde{\p}'=\tilde{\p}-\bm{m} = (0,5,0,2,3) \in \P_{5,2}.\) The same move $\bm{m}$  can be used  
 to move from $\p$ to $\p' = (3,8,3,5,6)$; in fact, it can be used for any position of the form $\p+c \cdot\z$.
\end{example}

We now have gathered all the tools necessary to establish how we use invariant vectors to prove the patterns of the \P-positions. \\

\subsubsection{Existence of a move from $\p \notin \P$ to $\p' \in \P$}

 Assume that the game $G$ cannot be reduced via a merge reduction and that we have identified a set of invariant vectors of the proposed set $\P_G$. (Invariance for a set $S$ of game positions can be verified by checking that the conditions that determine the set $S$ are still true for $\p+c\cdot \bm{z}$, which is easy to do when the conditions consist of a set of equations.) 

We first reduce the position $\p$ to a position $\tilde{\p}=\p-\sum_{i=1}^l {c_i\z_i}$ with carefully chosen coefficients $c_i$. Typically, the invariance reduction of the position creates a number of different cases to be considered, based on the number and relative location of the zeros of $\tilde{\p}$. For each sub-case, we then reduce the game $G$ using zero (and possibly) merge reductions to arrive at a sub-game $G_j$ that (hopefully) is already solved.  Since $\tilde{\p} \notin\P_{G_j}$, a move to a \P-position of the reduced game exists. We can use this move to arrive at a position $\p'\in \P_G$, for each of the sub-cases.  Figure~\ref{fig:IRP} visualizes the process.  

Before we give the details of this process and argue that it produces the desired result, we introduce some useful notation. Let ${I}_{\z}=\{i|z_i=1\}$ be the set of indices at which a zero-one vector $\z$ has ones (and hence, the indices of the stacks affected by  $\z$). We refer to any such index as a {\em 1-index of $\z$} and define the \emph{indicator minimum of $\z$ with respect to $\p$}, $\im(\z;\p)=\min_{I_{\z}}(p_i)$, as the minimal stack height among the stacks affected by $\z$. \\

\noindent {\bf \IRP}\\
Let $\{\z_1,\ldots,\z_l\}$ be a set of invariant vectors of the proposed set $\P_G$ and assume that $\p \notin \P_G$.
\begin{enumerate}
    \item (Invariance reduction) Set $\p_{(0)}=\p$ and iteratively compute  $c_i=\im(\z_i;\p_{(i-1)})$ and $\p_{(i)}=\p_{(i-1)}-c_i \z_i$ for $i=1,\ldots,l$. At each iteration, either $\p_{(i-1)}$ already has a zero at one of the 1-indices of $\z_i$ (in which case $c_i=0$) or a zero is created at a new location in $\p_{(i)}$. Applying this process, $\tilde{\p}:=\p_{(l)}$ has at least one zero. When performing the reduction manually, then the process can be sped up in two ways (see Example~\ref{ex:inv=P}): 
    \begin{itemize}
        \item[-] Disregard any invariant vectors $\z_i$ for which $\p$ has a zero at one of the  1-indices of $\z_i$, since that would result in $c_i=0$.
        \item[-] Process any vectors $\z_i$ that have no common 1-indices in a single iteration since the reductions in stack heights from their individual iteration have no affect on the stacks reduced by the other invariant vectors. 
    \end{itemize}
If Proposition~\ref{prop:invPpos} does not apply, then we look at each of the subspaces of positions induced by the number and relative location of the zeros. Each subspace is closed with regard to moves because the locations of the zeros are fixed, and therefore each subspace (case) can be considered separately. Note that the reduced position $\tilde{\p}$ is still a position in the original game $G$, and invariance implies that $\tilde{\p} \notin \P_G$. We now reduce the game in each of the subspaces.

\item (Subspace game reduction) For each of the subspaces that arise from the invariance reduction, we perform a zero reduction and possibly a merge-reduction to arrive at an isomorphic sub-game $\hat{G}$.   By Remark~\ref{rem:zero red} and Lemma~\ref{lem:merge red} we know that the associated position $\hat{\p}$ is an \N-position of $\hat{G}$, therefore a move $\hat{\bm{m}}$ from  $\hat{\p}$ to $\hat{\p}' \in \P_{\hat{G}}$ exists. 

\item (Winning move)  If the solution of $\hat{G}$ is known, then we can identify the move $\hat{\bm{m}}$ and,  by unraveling the game reduction steps, the corresponding move $\bm{m}$ in game $G$. If only a zero reduction was performed, then the move $\bm{m}$ is defined as  $m_i=0$ if $\tilde{p}_i=0$ and $m_i=\hat{m}_i$ if $\tilde{p}_i>0$. If the zero reduction was followed by a merge reduction, then the move $\bm{m}$  is defined as follows:
    $$m_i=\left\{\begin{array}{cl}
    0 & \text{if } \tilde{p}_i=0\\
    \hat{m}_i & \text{if }  \tilde{p}_i>0 \text{ and }  i \neq v^*\\
     \end{array} \right.$$
    where $v^*$ denotes the index of the merged stack. For the stacks that were merged, ``distribute" the total amount of tokens $\hat{m}_{v^*}$ to be removed across the stacks that were combined into vertex $v^*$.  The resulting move is not unique, unless $\hat{m}_{v^*}$ equals the sum of tokens of the merged stacks. A simple process to distribute the tokens is to remove as many tokens as possible from the smallest merged stack. If there are still tokens to be removed, remove as many as possible from the next smallest stack, and so on, until all $\hat{m}_{v^*}$ tokens have been removed. 
    \item (Finding the option) Since the sub-games are equivalent to the game $G$, the move $\bm{m}$ leads from $\tilde{\p}\notin \P_G$ to a position $\tilde{\p}'\in \P_G$.  By Lemma~\ref{lem:inv}, $\bm{m}$ is also a move from $\p\notin \P_G$ to $\p' \in \P_G$. Thus $\p'=\p-\bm{m}$.
\end{enumerate}

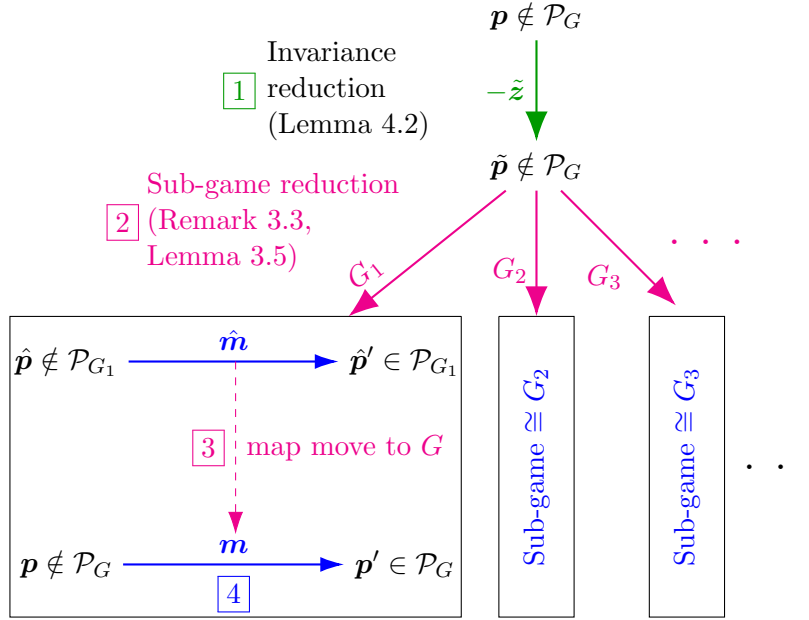
\begin{figure}[h!]
    \centering
  
\begin{tikzpicture}[font=\small]
\node at (0,6) (p1) {$\p \notin \P_G$};
\node at (0,4) (p2) {$\tilde{\p} \notin \P_G$};
\node at (2,2) (p4) {};

\draw[thick,green!60!black,-{Latex[length=4mm,width=3mm]}] (p1) -- node[left] {$-\tilde{\z}$} (p2);

\draw[thick,magenta,-{Latex[length=4mm,width=3mm]}] (p2) -- node[pos=0.55,below, xshift=-7.5pt] {$\tiny G_3$} (p4);

\draw[thick,magenta,-{Latex[length=4mm,width=3mm]}] 
    (p2) -- node[pos=0.45,below,sloped,rotate=90,xshift=-10pt] { $\tiny G_2$} (0,2);

\draw[]  (-1,2)-- (-1,-2)--(-7,-2)--(-7,2)--(-1,2);    
\draw[]  (-0.5,2)-- (0.5,2)--(0.5,-2)--(-0.5,-2)--(-0.5,2);
\draw[]  (1.5,2)-- (2.5,2)--(2.5,-2)--(1.5,-2)--(1.5,2);

\draw[thick,magenta,-{Latex[length=4mm,width=3mm]}] (p2) -- node[pos=0.8,above,sloped] {$\tiny G_1$} (-2.5,2);
\node[magenta,align=left] at (-3.5,3.25) {Sub-game reduction\\(Remark 3.3,\\Lemma 3.5)};
\node[align=left] at (-2.5,5) {Invariance\\reduction\\(Lemma 4.2)};

\node at (3.25,0)() {\bf{.~~.~~.}};
\node[magenta] at (2.25,3) () {\bf{.~~.~~.}};
\node[blue, rotate=90] at (0,0) () {Sub-game $\cong G_2$};
\node[blue, rotate=90] at (2,0) () {Sub-game $\cong G_3$};
\node[green!60!black] at (-3.95,5) () {\boxed{1}};
\node[magenta] at (-5.5,3.25) () {\boxed{2}};

\node at (-6.25,1.4) (p5) {$\hat{\p} \notin \P_{G_1}$};
\node at (-1.75,1.4) (p6) {$\hat{\p}' \in \P_{G_1}$};
\draw[blue,thick,-{Latex[length=3.2mm,width=2mm]}] (p5) -- node[above] {$\hat{\bm{m}}$}  (p6);

\node at (-6.25,-1.3) (p7) {$\p \notin \P_{G}$};

\node at (-1.75,-1.3) (p9) {$\p' \in \P_{G}$};
\draw[blue,thick,-{Latex[length=3.2mm,width=2mm]}] (p7) -- node[above]{$\bm{m}$}  node[below] {\boxed{4}}(p9);

\draw[dashed,magenta,-{Latex[length=3mm,width=2mm]}] (-4,1.4) -- node[magenta,right]{map move to $G$} node[magenta,left,pos=0.5] {\boxed{3}}(-4,-0.9);
\end{tikzpicture}
    
    \caption{Visualization of the \IRP. Boxed numbers reference the steps of the \IRP.}
    \label{fig:IRP}
\end{figure}

In Examples~\ref{ex:irp} and~\ref{ex:irp2} of Section~\ref{sec:pathnim}, we will provide numerical examples on how to map the move $\hat{\bm{m}}$ of the reduced game back to a move $\bm{m}$ of the original game for the game $H$, which is introduced in that section. \\

As indicated in the description of the \IRP, each iteration creates a zero at one of the 1-indices of the currently applied invariant vector, not necessarily an additional one. If the order in which we apply the invariant vector is changed, then the corresponding reduced positions $\tilde{\p}$ may differ. However, they will belong to the same subspace of positions.  Table~\ref{tab:inv_var2} shows a few numerical examples of this phenomenon for the game $H$ (see Theorem~\ref{thm:H}), which has invariant vectors $\bm{z}_1=(1,1,0,1,0,1)$ and  $\bm{z}_2=(1,0,1,0,1,1)$. In each row of the table, we list the position $\p$, followed by the first iterate $\p_{(1)}$ and the final reduced position $\tilde{\p}$, for each of the two possible orders of applying the invariant vectors. The last column indicates the case in the proof of Theorem~\ref{thm:H} that $\tilde{\p}$ belongs to. In all instances, the two resulting reduced positions belong to the same subspace (case) determined by the location of their zero(s), even though their stack heights may differ and additional zeros may be present. For this game, reduced positions that differ fall into case 1 because one of two conditions holds: either one of the two smallest stacks of $\p$ is at location $a$  or the sum of the two smallest stacks of $\p$ exceeds  $a$.

\begingroup
\setlength{\tabcolsep}{8pt} 
\renewcommand{\arraystretch}{1.25} 
\begin{table}[h!]
    \centering
    \scalebox{0.90}{
    \begin{tabular}{|c||c|c!{\vrule width 1.2pt}c|c||c|} \hline
      & \multicolumn{2}{c!{\vrule width 1.2pt}}{$\z_1$ first, then $\z_2$} & \multicolumn{2}{c||}{$\z_2$ first, then $\z_1$} &      \\ \hline

       $\p$  & $\p_{(1)}$ & $\tilde{\p}$ &  $\p_{(1)}$ & $\tilde{\p}$ & Case \\\hline\hline

     $(5,2,7,8,9,6)$      &$(3,0,7,6,9,4)$  & $(\bm{0},\bm{0},4,6,6,1)$ & $(\bm{0},2,2,8,4,1)$ & no change & 1 \\\hline
     
  $(3,5,6,3,9,10)$ & $(\bm{0},2,6,\bm{0},9,7)$  & no change & $(\bm{0},5,3,3,6,7)$ & no change   &  1\\ \hline

    $(4, 3, 9, 12, 2, 8)$  &$(1, \bm{0}, 9, 9, 2, 5)$ & $(\bm{0}, \bm{0}, 8, 9, 1, 4)$ & $(2, 3, 7, 12, \bm{0}, 6)$ & $(\bm{0}, 1, 7, 10, \bm{0}, 4)$ & 1 \\ \hline\hline
    
    $(8,2,4,6,5,7)$      &$(6,0,4,4,5,5)$  & $(2,\bm{0},\bm{0},4,1,1)$ & $(4,2,0,6,1,3)$ & $(2,\bm{0},\bm{0},4,1,1)$& 2 \\\hline 

     $(5,4,2,2,3,7)$      &$(3,2,2,0,3,5)$  & $(1,2,\bm{0},\bm{0},1,3)$ & $(3,4,0,2,1,5)$ & $(1,2,\bm{0},\bm{0},1,3)$&  3\\\hline
     
   $(7,2,7,5,3,8)$   & $(5,0,7,3,3,6)$ & $(2,\bm{0},4,3,\bm{0},3)$ &  $(4,2,4,5,0,5)$& $(2,\bm{0},4,3,\bm{0},3)$& 4\\\hline

    \end{tabular}}
    \caption{Positions $\p$ of game $H$, with their respective invariance-reduced positions $\tilde{\p}$, arrived at by applying the invariant vectors in different order.  Zero entries of $\tilde{\p}$ are indicated in bold for easy reference to the case in the proof of Theorem~\ref{thm:H}, which is listed in the last column.}
    \label{tab:inv_var2}
\end{table}
\endgroup

\subsubsection{There is no 
move from $\p \in \P$ to $\p' \in \P$}

We can also use the invariance reduction process for some parts of the proof that there is no move from a  \P-position to another \P-position. As indicated before, the invariance reduction process produces subsets of positions that have a specified number of zeros at given relative locations. Each such set of positions is a closed under legal moves because the locations of zeros do not change. Thus, if the invariance reduction process leads to a solved equivalent game, then we can assert that for these subspaces of \P-positions, there cannot be a move to another \P-position. 

However, the moves in these subspaces do not account for all possible moves from a given position $\p \in \P$ because the minimum value of the reduced position is zero and thus any move keeps the location of the minimum fixed. This excludes moves to positions $\p'$ where a new, smaller minimum is at a different location. These types of moves need to be analyzed the ``old-fashioned way". Example~\ref{ex:Ppos52} illustrates this problem. \\

\begin{example} \label{ex:Ppos52} The \P-positions of \CN{5}{2} are given by $\P = \{\p=(M, m, c, d, m) \mid M + m = c + d \text{ and } M = \max(\p)\}$.  
The single invariant vector is $\z=\bm{1}$, so the reduced positions are of the form $(M-m,0,c-m,d-m,0)$.  For position $\p=(5,1,4,2,1)$, the move $\bm{m}=(0,0,2,1,0)$ to $\p'=(5,1,2,1,1)$ has a corresponding move for the reduced position, namely from $\tilde{\p}=(4,0,3,1,0)$ to $\tilde{\p}'=(4,0,1,0,0)$. However, the move $\bm{m}=(5,1,0,0,0)$ from $\p$ to $\p'=(0,0,4,2,1)$ does not have a corresponding move in the subspace of the reduced position because we cannot move on the minimal stacks.  
\end{example}

\subsection{Efficient Play with Invariant Vectors}

When we play, the only thing that matters is to know the winning moves for a player in an \N-position (as the person in a \P-position can make any move, it does not matter). Given the many possible moves, how can we play efficiently in the sense of having just a few facts that we need to know to be able to find the winning move?

Invariance will assist here too. For one, the structure of the invariant vectors, specifically if there are only a few different ones, can suggest a useful way to describe the set of \P-positions. For example, the \P-positions of \CN{5}{2} were listed as $(m,M,m,c,d)$ with $m$ being the minimum, $M$ being the maximum, and $m+M=c+d$. The only invariant vector is $\z=\bm{1}$. Since we are subtracting the same number of tokens from each stack when applying $\z$, this suggests to represent the \P-positions instead as $(m,m+x_1,m,m+x_2,m+x_3)$ with $x_1=x_2+x_3$. This makes it quite easy to see the structure of the reduced positions $\tilde{\p}=\p-m\z =(0,x_1,0,x_2,x_3)$. After zero and merge reductions, the sub-game is congruent to \Nim~on two stacks, with positions $\hat{\p}=(x_1,x_2+x_3)$; the winning move is to equalize the two stacks. 

Other examples of expressing the \P-positions in alignment with the invariant vectors are shown in Example~\ref{ex:CN63lc}, where we have iteratively built up a representation of the \P-positions as the linear combinations of the invariant vectors. Another example is the  game \CN{8}{3} discussed in Section~\ref{sec:CN_8-3}, where the two invariant vectors consisting of alternating zeros and ones lead to the representation of the \P-positions in terms of the minima of the two sets of stacks that are distance two apart. 

In the game \CN{5}{2} it is easy to remember the winning move because there is just one sub-game created by the single invariant vector. For games with multiple invariant vectors, we typically have to deal with several sub-games and need to know how to play in each of the cases.  Due to the \IRP, for each position $\p$ it is sufficient to just look at its invariance-reduced positions $\tilde{\p}$ (classified by the location and number of zeros) and know the associated reduced position $\hat{\p}$ and the sub-game to which $\hat{\p}$ belongs. The winning moves that we identify for the sub-game lead to winning moves for any position that reduces to the particular  sub-case. This greatly reduces the types of moves we need to keep track of.  Table~\ref{tab:HCheat}  shows a ``cheat sheet" for the game $H$ discussed in Section~\ref{sec:pathnim}.

\section {\P-positions of PathNim and Game $H$} \label{sec:pathnim}

In this section we derive the \P-positions of both the family of \ruleset{PathNim} games when play is allowed on at least half of the stacks and  of the game  $H \defeq \{\{a,b,c\},\linebreak[1]\{b,c,d\},\linebreak[1]\{c,d,e\},\linebreak[1]\{d,e,f\},\linebreak[1]\{a,f\}\}$. The game $H$ can be visualized as the game \PN{6}{3}~with one additional move, the move on the two ends. It will play an important role as a sub-game of \CN{7}{3} in the proofs of our main results.   

The result given in Theorem~\ref{thm:P-pos path} on the \P-positions of certain \ruleset{PathNim}~games  was stated and proved in~\cite{ES1996}  using the full machinery of simplicial complexes. We give a proof that just uses combinatorial game theory.

\begin{theorem} \label{thm:P-pos path} If $k \geq \lceil \frac{n}{2}\rceil$, that is, when play is allowed on at least half of the stacks, then  $$\P_{n,k}=\{(a_1,\ldots,a_\ell,\underbrace{0,\ldots,0}_{k-1},b_1,\ldots,b_m)\ | \ \ell+m=n-k+1, \sum_{i=1}^{\ell} a_i=\sum_{j=1}^m b_j, \min\{\ell,m\}\geq 1\}.$$
In particular, $\P_{n,n}=\{(0,\ldots,0)\}$.
    \end{theorem}
    
\begin{proof} 
 Let $\p \in \P_{n,k}$. The condition  $\ell+m=n-k+1$ ensures that there are (at least) $k-1$ consecutive zeros. Play is restricted to one of the sets of stacks on the left and right of the sequence of consecutive zeros, of which there are at most $k$ on either side, and any move to $\p'$  destroys the equality of the two sums, thus $\p'\notin \P_{n,k}$. Because $k \geq \lceil \frac{n}{2}\rceil$, we cannot have two instances of a string of $k-1$ zeros that are separated by a non-zero stack (at most, they can be adjacent, forming a sequence of $2(k-1)$ consecutive zeros). In any such string of more than $k-1$ zeros, it does not matter where we assume the designated $k-1$ zeros to be, as the remaining zeros do not contribute anything to the sums on the left and right. 
 
 To show that there is always a move from $\p\notin \P_{n,k}$ to a position $\p'\in \P_{n,k}$, we consider two cases: If $\p$ contains a sequence of at least $k-1$ consecutive zeros with non-zero stacks on either side of the string of zeros, then equalize the two sums. Since the non-zero stacks consist of at most $k$ stacks on either side, an equalizing move can be made.  
 
 Now assume we do not have a string of zeros with two non-zero sums on either side. We need to find a location where we can reduce $k-1$ consecutive stacks to zero {\bf and} equalize the left and right sums.  Let $ \p = (p_1,  \ldots, p_n) \notin \P_{n,k}$. We define the \textit{left-hand sum} and \textit{right-hand sum} functions for $\p$ as 
 \[
L(i) = p_1 + \dots + p_i \quad \text{and} \quad R(i) = p_{i+k} + \dots + p_n 
\]
and will use $L'(i)$ and $R'(i)$ for the corresponding sums of an option $\p'$ of $\p$. First suppose there exists an index $i$ such that $L(i) = R(i)$.  Since $\p \notin \P_{n,k}$, at least one of the entries $p_{i+1}, p_{i+2}, \dots, p_{i+k-1}$ must be non-zero, so the move to $\p'=(p_1,\ldots, p_i,0,\ldots,0,p_{i+k},\ldots, p_n) \in  \P_{n,k}$ is legal.

Now assume that \( L(i) \ne R(i) \) for all \( i \). We look at three sub-cases and only indicate which stacks change in the move to $\p'$ and how they change. Any stack not mentioned is assumed to remain the same. 
\begin{itemize}
    \item \textbf{Case 1:} \( p_1=L(1) > R(1) \). We can move to $\p'\in \P_{n,k}$ where 
    \[
    p_1' =  R(1)  \text{ and} \phantom{x}p_2' = p_3' = \dots = p_k' = 0.
    \]
     \item \textbf{Case 2:} $L(n-k) < R(n-k)=p_n$.
    In this case, we move to $\p'\in \P_{n,k}$ where 
    $$ 
     p_{n-k+1}' =   p_{n-k+2}' = \dots = p_{n-1}' = 0 \phantom{x} \text{and } p_n' =  L(n-k).
    $$
    \item \textbf{Case 3:} Now assume that \( L(1) < R(1) \) and \( L(n-k) > R(n-k) \). Then there must be an index $i$ such that 
    $$L(i) < R(i) \quad \textnormal{and} \quad L(i+1) > R(i+1),$$ as otherwise, both \( L(i) < R(i) \) and \( L(i+1) < R(i+1) \) hold for all values of $i$. Together with  the assumption $L(1) < R(1)$ this would (iteratively) imply  that $L(n-k) < R(n-k)$, a contradiction. Hence, such an index $i$ must exist. Now we construct the winning move. We first set $p_{i+2}=\ldots=p_{i+k-1}=0$ and then use $p_{i+1}$ and $p_{i+k}$ to equalize the sums and to create the missing zero for the string of $k-1$ consecutive zeros. 
\begin{itemize}
    \item If $L(i)<R(i+1)$,  we let $ p'_{i+1}=R(i+1)-L(i)<L(i+1)-L(i)=p_{i+1}$. Then $R'(i+1)=R(i+1)$ and $L'(i+1)=L(i)+p'_{i+1}=R(i+1)$, so the two sums of $\p'$ are equal.
    \item If $L(i)>R(i+1)$,  we let $p_{i+k}'=L(i)-R(i+1) <R(i)-R(i+1)=p_{i+k}$. Then $L'(i)=L(i)$ and $R'(i)=p'_{i+k}+R(i+1)=L(i)$, so the two sums are once more equal.
\end{itemize}  
 \end{itemize}
This completes the proof. Figure~\ref{fig:case3move} visualizes the quantities involved. 
\end{proof}   
    
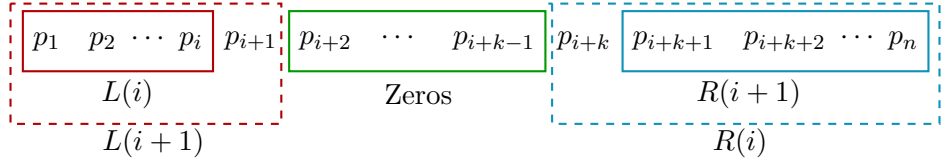
\begin{figure}[!hbt]
    \centering    

    \begin{tikzpicture}[font=\small, node distance=0mm]

\node[draw=red!70!black, thick, minimum height=0.8cm, minimum width=2.5cm, anchor=west] (L) at (0,0) {$p_1 \quad p_2 \ \cdots\ p_i$};

\node[anchor=west] (mid1) at (L.east) {$p_{i+1}$};

\node[draw=green!60!black, thick, minimum height=0.8cm, minimum width=2.5cm, anchor=west] (zeros) at (mid1.east) {$p_{i+2} \, \, \, \, \, \cdots  \, \, \, \, \, \ p_{i+k-1}$};

\node[anchor=west] (mid2) at (zeros.east) {$p_{i+k}$};

\node[draw=cyan!70!black, thick, minimum height=0.8cm, minimum width=2.5cm, anchor=west] (R) at (mid2.east) {$p_{i+k+1} \quad p_{i+k+2} \ \cdots \ p_n$};

\node[midway, xshift=40pt, yshift=-20pt]{$L(i)$};
\node[midway, xshift=150pt, yshift=-20pt]{Zeros};
\node[midway, xshift=275pt, yshift=-20pt]{$R(i+1)$};

\draw[thick, dashed, red!70!black] (-0.15,0.5)--(-0.15,-1.1)--(3.45,-1.1)--(3.45,0.5)--(-0.15,0.5);

\node[] at (1.75, -1.35){$L(i+1)$};

\draw[thick, dashed, cyan!70!black] (7.05,0.5)--(7.05,-1.1)--(12.2,-1.1)--(12.2,0.5)--(7,0.5);

\node[] at (9.55, -1.35){$R(i)$};

\end{tikzpicture}
    \caption{Quantities involved in the move to  $\p' \in \P_{n,k}$ in case 3 of the proof of Theorem~\ref{thm:P-pos path}.}
    \label{fig:case3move}
\end{figure}

We now turn our attention to the game $H$, defined in Theorem~\ref{thm:H}. The game $H$ is a sub-game of both \CN{7}{3} and \CN{9}{4}, so Theorem~\ref{thm:H} is a crucial piece in the derivation of the \P-positions of both games. We give the solution for \CN{7}{3} in Section~\ref{sec:CN_7-3}.  While we also have the result for the \P-positions of \CN{9}{4},  our current proof is tedious and too long for inclusion here. The solution for \CN{9}{4} will be a focal point in a subsequent paper.

\begin{theorem} \label{thm:H} Let $H\defeq\{\{a,b,c\},\{b,c,d\},\{c,d,e\},\{d,e,f\},\{a,f\}\}$. Then     
     $$ \P_H= \{(a,b,c,d,e,f) \mid a \le f, \,a+b+c=d+e+f,\, 
        a=d+min(c, e)\}. $$
\end{theorem}

For ease of readability, we will adopt a simplified notation for the move sets in the proofs that follow. Instead of the set notation $\{a,b,c\}$ we will use $(abc)$.

\begin{proof}
 To show that there is no move from $\p \in \P_H$ to $\p' \in \P_H$, we consider the impacts of the various allowed moves on the two equations, which we will refer to as triple sum and min sum condition.  A move on either $(abc)$ or $(def)$ affects one side of the triple sum equality and not the other, so $\p' \notin \P_H$. A move on $(af)$ needs to reduce both stacks by the same amount to maintain the triple sum condition,  so $a' \leq f'$. The move affects only one side of the min sum condition, so once more, $\p' \notin \P_H$. For the move $(bcd)$, to maintain the triple sum condition, there has to be a change in $d$ that is matched by the combined change of $b$ and $c$. Such a move only changes the right-hand side of the min condition, so $\p' \notin \P_H$. A similar argument applies for the move $(cde)$. This completes this part of the proof.

Now we show that from $\p=(a,b,c,d,e,f) \notin \P_H$, there is a move to $\p' \in \P_H$ using the \IRP. It is immediate from the description of $\P_H$ that the two zero-one vectors $\bm{z}_1=(1,1,0,1,0,1)$ and $\bm{z}_2=(1,0,1,0,1,1)$ are invariant for $\P_H$. Let $c_1=\im(\z_1;\p)$, $\p_{(1)}=\p-c_1 \z_1$, $c_2=\im(\z_2;\p_{(1)})$, and $\tilde{\p}=\p_{(1)}-c_2 \z_2$. This process will create either a single zero (necessarily at stack $a$) or at least two zeros in $\tilde{\p}$. The configuration of the zeros falls into one of the four cases (up to symmetry) below. In each case, we reduce to a sub-game that is solved, so we know there is a legal move to a position $\p' \in \P_H$ by the \IRP. We just state the reduced game. Note that the cases relate to stack heights of $\tilde{\p}$. 
\begin{itemize}
\item {\bf Case 1}: $a=0$. 
Zero reduction gives $H\cong \{\{b,c,d\},\{c,d,e\}, \{d,e,f\}\} \cong$  \PN{5}{3}. 
\item {\bf Case 2}: $b=c=0$ (equivalent to $d=e=0$). Zero reduction followed by a merge reduction gives  $H \defeq \{\{ a,f\},\{d,e,f \} \}\cong \{\{ a,f\},\{f, d+e \}\}\cong$ \PN{3}{2}.
\item {\bf Case 3}: $c=d=0$.
Here $H$ reduces to $\{\{a,b\},\{e,f\},\{a,f\}\}\cong \{\{b,a\},\allowbreak \{a,f\},\allowbreak \{f,e\}\}\allowbreak \cong$ \PN{4}{2}.
\item {\bf Case 4}: $b=e=0$. 
Zero reduction gives $H \defeq \{\{a,c\},\{c,d\},\{d,f\},\{f,a\}\}\cong$ \CN{4}{2}.
\end{itemize}

In each of the sub-cases, identify the winning move  $\hat{\bm{m}}$ for the reduced game and then create from it the winning move $\bm{m}$ as described in the \IRP. Finally compute $\p' \in \P_G$. This completes the proof.
\end{proof}

Examples~\ref{ex:irp} and~\ref{ex:irp2} show the details the process of finding the actual position $\p' \in \P_H$, for positions that fall into cases 1 and 2. Figure~\ref{fig:case1move} visualizes the process. 

\begin{example}\label{ex:irp}
First we consider a position with $\min(\p)=a$, which implies that $c_1=a$ and $c_2=0$.  Let $\p=(2,6,11,8,3,12)\notin \P_H$. Then $\tilde{\p}=\p - 2\cdot \z_1=(0,4,11,6,3,10)$. The zero reduction eliminates the first stack, so $\hat{\p}=(4,11,6,3,10)$. The reduced game is \ruleset{PathNim} \PN{5}{3}, so we need to create two consecutive zeros with equal sums on either side. One such move is $\hat{\bm{m}}=(0,5,6,3,0)$, which leads to $\hat{\p}'=(4,6,0,0,10)\in \P_{5,3}$. Reversing the steps used to obtain the sub-game, we add back the zero stack. It did not change, so $\bm{m}=(0,0,5,6,3,0)$. Using the same move $\bm{m}$ yields $\p'=\p-\bm{m}=(2,6,6,2,0,12)\in\P_H$. 
\end{example}

\begin{figure}[h!]
    \centering

\begin{tikzpicture}[font=\small]

\node at (0,3.5) (p) {$\p=(2,6,11,8,3,12) \notin \P_{H}~~~$};
\node at (11,3.5) (p1) {$~~~\p'=(2,6,6,2,0,12) \in \P_{H}$};

\node at (0,1.75) (p2) {$\tilde{\p}=(0,4,11,6,3,10) \notin \P_{H}~~~$};

\node at (0,0) (p4) {$~~\hat{\p}=(4,11,6,3,10) \notin \P_{5,3}~~~$};

\node at (11,0) (p5) {$~~~\hat{\p}'=(4,6,0,0,10) \in \P_{{5,3}}~~$};

\draw[blue,thick,-{Latex[length=3.2mm,width=2mm]}] (p) -- node[below] {$\bm{m}=(0,0,5,6,3,0)$} node[above,pos=0.5] {\boxed{4}} (p1);

\draw[thick,green!60!black,-{Latex[length=3.2mm,width=2mm]}] (p) -- node[right, pos=0.34] {$-2 \cdot z_1$} node[left,pos=0.34] {\boxed{1}} (p2);

\draw[thick,magenta,-{Latex[length=3.2mm,width=2mm]}] (p2) -- node[right, pos=0.34] {Zero reduction}(p4);

\node[magenta] at (-0.3,1.1) {\boxed{2}};

\draw[thick, blue,-{Latex[length=3.2mm,width=2mm]}] (p4) -- node[above] {$\hat{\bm{m}}=(0,5,6,3,0)$} (p5);

\draw[thick,magenta,dashed,-{Latex[length=3.2mm,width=2mm]}] (5.5,0.76) -- node[right,xshift=4pt,yshift=1pt, pos=0.5] {\boxed{3}}(5.5,3);

\end{tikzpicture}
    \caption{The steps of the \IRP~for a position $\p \notin \P_H$ with $\min(\p)=a$, which reduces to game \PN{5}{3} (case 1 in the proof of Theorem~\ref{thm:H}), using the position of Example~\ref{ex:irp}.}
    \label{fig:case1move}
\end{figure}
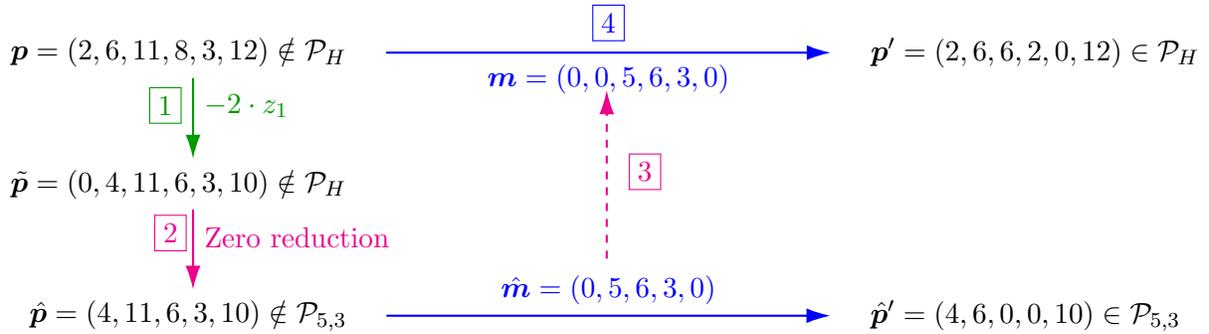

\begin{example} \label{ex:irp2}
    Now we consider a position where the reduction involves both a zero reduction and a merge reduction. Let $\p=(6,2,9,1,3,12) \notin \P_H$. Here $c_1=\min(6,2,1,12)=1$, $\p_{(1)}=\p-\z_1=(5,1,9,0,3,11)$, $c_2=\min(5,9,3,11)=3$, $\tilde{\p}=\p_{(1)}-3\z_2=(2,1,6,0,0,8)$. This position has $d=e=0$, so falls into the alternate version of case 2. We first perform the zero reduction, resulting in sub-game $\{\{a,b,c\},\{a,f\}\}$, which can be merge-reduced to $\{\{a,b+c\},\{a,f\}\}\cong$ \PN{3}{2} with positions $(b+c,a,f)$. Thus $\hat{\p}=(7,2,8)$. We need to create a zero with equal stack heights on the left and right, so the move is $\hat{\bm{m}}=(0,2,1)$, which leads to $\hat{\p}'=(7,0,7)\in \P_{3,2}$. This translates to no decrease in the $b$ or $c$ stacks, decrease of $2$ on stack $a$, and decrease of $1$ on the $f$ stack. Accounting for the zero stacks $d$ and $e$, which also have not changed, the corresponding move in the original game is   $\bm{m}=(2,0,0,0,0,1)$ and  we obtain $\p'=\p-\bm{m}=(4,2,9,1,3,11) \in \P_H$.
\end{example}

Table~\ref{tab:HCheat} provides a ``cheat sheet" for playing the game $H$.  For each of the cases in the proof of Theorem~\ref{thm:H}, it shows the reduced positions and the game in which we need to find a move. Theorems~\ref{thm:P-pos path} and~\ref{thm:sumCN} list the relevant \P-positions for the respective sub-games.

\begingroup
\setlength{\tabcolsep}{10pt} 
\renewcommand{\arraystretch}{1.25} 

\begin{table}[h!]
\centering
\begin{tabular}{|c|c|c|c|}
\hline
{\bf Case }              
& \multicolumn{1}{c|}{$\tilde{\p}$} 
& \multicolumn{1}{c|}{$\hat{\p}$} 
& \multicolumn{1}{c|}{{\bf Reduced Game}}      \\ \hline \hline

1  & $(\bm{0},b,c,d,e,f)$                   
    & $(b,c,d,e,f)$    & \PN{5}{3}     \\ \hline \hline

\multirow{2}{*}{2} & $(a,\bm{0},\bm{0},d,e,f) $                 
    & $(a,f,d+e)$  
    & \multirow{2}{*}{\PN{3}{2}  } \\ \cline{2-3}
    & $(a,b,c,\bm{0},\bm{0},f)$    & $(b+c,a,f)$   &    \\ \hline \hline

3   & $(a,b,\bm{0},\bm{0},e,f)$    & $(b,a,f,e)$   & \PN{4}{2}                  \\ \hline \hline

4   & $(a,\bm{0},c,d,\bm{0},f)$    & $(a,c,d,f)$    & \CN{4}{2}                  \\ \hline
\end{tabular}

\caption{Positions $\tilde{\p}$ of game $H$ after invariance reduction, corresponding position $\hat{\p}$ in associated zero- and merge-reduced sub-game, and the sub-game in which need to find a winning move from  $\hat{\p}$. Zeros entries of $\tilde{\p}$ are indicated in bold for easy reference to the case in the proof of Theorem~\ref{thm:H}.}
\label{tab:HCheat}
\end{table}
\endgroup

We now turn to the main results of this paper.

\section{The \P-positions of  \CN{7}{3}}\label{sec:CN_7-3}

\begin{theorem} \label{thm:P7-3} Let $\p=(a,b,c,d,e,f,g)$. The set of \P-positions for the game \CN{7}{3} is given by 
$$\P_{7,3}=\{ \p \mid a=\min(\p),~b\leq g,~a+b=e+\min(d,f),~b+c+d=e+f+g\}.$$
\end{theorem}

We will refer to the last two conditions as the min sum (MS) and the triple sum (TS)  conditions, respectively, and will use \MS~and \TS~when we refer to the corresponding conditions for $\p'$. In the proof, we will also use \LM~and \RM~to refer to the left- and right-hand side of \MS, respectively, and corresponding notation for \TS.

\begin{remark} \label{rem:ineqs} Note that the (TS) and (MS) conditions for \p~imply that $b \geq e$ and $g \geq d$. The first inequality follows from  $b=e+\min(d,f)-a \geq e$.  We also can derive that $g\geq b \geq \min(d,f)$. If $\min(d,f)=d$, then $g\geq d$ is immediate. If $\min(d,f)=f$, then by (MS), $a+b=e+f$. Using (TS) and $a+b=e+f$, we get that $c+d=a+g$, which implies $d\leq g$.
\end{remark}

\begin{proof}
First we show that there is no move from a position $\p\in\P_{7,3}$ to a position $\p'\in\P_{7,3}$.   There are two possibilities for $\p'$: the new minimum is at the $a$ stack or it is at some other stack. In the first case, the only play that maintains both the (MS) and the (TS) conditions is to play on stacks $b$ and $e$, reducing them both by the same amount. This requires play on four stacks which is not allowed. 

Now let's consider the possibility that $\min(\p')$ is not at the $a$ stack, so $\min(\p')<a$. We will consider each alternative location of the new minimum, and indicate stacks that may have been changed by a move resulting in this new minimum by a prime.  While five stacks will be listed with primes, only three consecutive stacks have actually changed; which ones depends on the specific move.  When considering the various locations of the new minimum, we need to adjust (rotate) the \TS~and \MS~conditions based on the location of the minimum and need to consider which direction (clockwise or counter-clockwise) is appropriate based on the relative stack heights of the two stacks adjacent to the new minimum. We consider two sub-cases, and will start with the clockwise orientation first for each location of the new minimum. The conditions regarding  the direction  may become part of the argument, but not always. Furthermore, $\min(\p')<a'\leq a$ since we assume that the new minimum is not at stack $a$.

\begin{itemize}
    \item $\min(\p')=b'$, $\p'=(a',b',c',d',e,f,g')$: If $c'<a'$, then for moves $(abc)$ or $(bcd)$ we have that $\LM=b'+c'<2a\leq f+\min(g,e)=\RM$, while this case cannot happen for move $(gab)$, since $c=c'<a'\leq a$, a contradiction. If $a'\leq c'$, then for moves $(gab)$ or $(abc)$, we have that $\LM =a'+b'<a+b=e+\min(d,f)=\RM$. If play is on $(bcd)$ with $d'<d$ (otherwise move is covered by $(abc)$), then $\MS$ is given by $a+b'=e+\min(d',f)$. If $\MS$ is true, then we must have that  $e=a$ and $d'=b'$.  Using $a=e$, (TS), and $b \geq e$, we have that  $\LT=a+g+f= e+g+f= b+c+d>e+c'+d'=\RT$, so  at least one of the two conditions fails. Thus,  $\p' \notin \P_{7,3}$. 
    \item $\min(\p')=c'$, $\p'=(a',b',c',d',e',f,g)$:  First assume that $ d' \leq b' $. Since  $ c' < a' $  and $d' \leq d \leq g$, it follows that $\LM = c' + d' < a' + g = g + \min(a', f) = \RM$. Now, assume instead that \( d' > b' \). Then the required equations are $c' + b' = f + \min(g, e') = f + e'$ (because $g \geq b\geq e$) and $b' + a' + g = d' + e' + f$. Combining $\MS$ and $\TS$, we obtain that $a' + g = d' + c'$. This implies $g < d'$, which contradicts that $g \geq d$ by Remark~\ref{rem:ineqs}. Therefore, at least one of the equations $\MS$ or $\TS$ is invalid, so $\p' \notin \P_{7,3}$.
    
    \item $\min(\p')=d'$, $\p'=(a,b',c',d',e',f',g)$: First, assume that $c' \geq e'$.  For moves on $(cde)$ or $(def)$, we have that $\LM = d' + e' < \min(d, f) + e = a + b = a+\min(g,b)=\RM$, where we have used that $d'<a\leq \min(d, f)$ and the (MS) condition. For the move on $(bcd)$ with $b'<b$, $\LT=e+f+g =b+c+d > b'+c'+a=\RT$. Now suppose that $c' < e'$. Since $c' < e' \leq b \leq g$, we have $\LM = d' + c' < \min(f', a) + g =\RM$. Thus, $\p' \notin \P_{7,3}$. 
    
    \item $\min(\p')=e'$, $\p'=(a,b,c',d',e',f',g')$: First suppose $f' \leq d'$. Then for moves on  $(def)$ and $(efg)$, $\LM = e' + f' < e + \min(d, f) = a + b = b+\min(a,c)=\RM$.  For the move on $(cde)$ with $c'<c$ and $f'=f$, we have that $\MS$ becomes $e' + f=b+\min(a,c')$. If $a\leq c'$, then $\MS$ fails in the same way as for the other two moves. If on the other hand, $c'<a$ and we assume that $\MS$ holds, then $\TS$ becomes $ f+g+a=b+c'+d'=e'+f+d'$. Since $e'<a$, this implies that $g< d'$, a contradiction. Therefore, at least one of the two conditions fails. 
    
    Now assume that $f' > d'$. If play is on $(def)$ or $(cde)$, then $\LM = e' + d' <  e + \min(d, f) = a + b = a+\min(b,g)=\RM$. For the move on $(efg)$ with $g' < g$ we have  that $\LT = b + c + d = e + f + g \geq a + f + g > a + f' + g' = \RT$, using (TS) for the first equality.  Hence, in all sub-cases, $\p' \notin \P_{7,3}$. 
    \item $\min(\p')=f'$, $\p'=(a',b,c,d',e',f',g')$:   First suppose $e' > g'$. Then $\LM = f' + g'$, and $\RM = c + \min(b, d')$. If $b \leq d'$, then $\RM = c + b \geq a + b = e + \min(d, f) > e' + f' > g' + f' = \LM$.  On the other hand, if $b > d'$, then \MS~and\TS~are given by $f' + g' = c + d'$, and $e' + d' + c = g' + a' + b$. Combining these gives $e' + f' = a' + b$, which implies $e' > b$, contradicting  that $b \geq e \geq e'$ by Remark~\ref{rem:ineqs}. Now assume $e' \leq g'$. Then, $\LM = f' + e' < a' + b = b + \min(a', c) = \RM$ since $f' < a'$  and $e' \leq e \leq b$. Therefore, \( p' \notin \P_{7,3} \) in all sub-cases.
    \item $\min(\p')=g'$, $\p'=(a',b',c,d,e',f',g')$: If we assume that $g'$ is a unique minimum, then if $a' \leq f'$, we have that $\LM=g'+a'\leq g'+a < \min(c,e')+g=\RM$, because $g'<e'$ due to uniqueness, and $g'<a\leq c$. Similarly, if $a'>f'$, then $\LM=g'+f'< g'+a<\min(d,b')+c =\RM$,  irrespective of the move that was made.  If $g'$ is not unique, then a second minimum occurs at one of the other stacks, and those cases have been discussed above.   Thus, $\p' \notin\P_{7,3}$. 
\end{itemize}

\noindent This concludes the proof that there is no move from $\p \in \P_{7,3}$ to $\p' \in \P_{7,3}$. We next show that for any $\p \notin \P_{7,3}$ there is a move to $\p' \in \P_{7,3}$. It is easy to check that  $\z=\bm{1}$ is an invariant vector for $\P_{7,3}$. Therefore, by Lemma~\ref{lem:inv}, we can restrict ourselves to finding a move from $\tilde{\p}=\p-\min(\p)\cdot \z$, that is, we can assume that $\min(\p)=0$.  Performing a zero reduction results in $\tilde{G}\defeq \{\{b,c,d\},\{c,d,e\},\{d,e,f\},\{e,f,g\},\allowbreak \{g,b\}\}\cong H$, which is solved (see Theorem~\ref{thm:H}), so we can find  a move from any $\p \notin \P_{7,3}$  to $\p' \in \P_{7,3}$.  This completes the proof. 
\end{proof}
We can use Table~\ref{tab:HCheat} to play \CN{7}{3}. Reduce a position $\p=(a,b,c,d,e,f,g)$ of CN(7,3) by subtracting the minimum $a$ from all stacks. On the stacks $(b,c,d,e,f,g)$ perform the invariance reductions for the game $H$. Play the move from game $H$ on all but the minimal stack, which remains unchanged.  Example~\ref{ex:CN73play} provides a numerical example of that process.

\begin{example} \label{ex:CN73play} Let $\p = (3,5,9,14,11,6,15) \notin \P_{\CN{7}{3}}$. Performing the invariance reduction with vector $\z=\bm{1}$ subtracts the minimum value from all stacks and results in $\tilde{\p} = \p - 3 \cdot \bm{z} = (0,2,6,11,8,3,12)$. 
Applying zero reduction yields $\hat{\p} = (2,6,11,8,3,12) \notin \P_{H}$, which is the  position we started from in Example~\ref{ex:irp}. We obtained the move $\hat{\bm{m}} = (0,0,5,6,3,0)$ from $\hat{\p}$ to a position in $\P_{H}$. Lifting this move to the original game (by adding a zero for stack $a$) gives $\bm{m} = (0,0,0,5,6,3,0)$ and thus 
$\p' = \p - \bm{m} = (3,5,9,9,5,3,15) \in \P_{\CN{7}{3}}$.
    \end{example}

\section{The \P-positions of \CN{8}{3}}\label{sec:CN_8-3}
Before defining the losing set for \CN{8}{3}, we first introduce some general assumptions for an arbitrary 8-tuple $\p = (p_0,p_1, p_2, \ldots,p_7)$. Define the \emph{square minima}
$$
a = \min(p_0,p_2,p_4,p_6) \quad \text{and} \quad b = \min(p_1,p_3,p_5,p_7)$$ and assume without loss of generality that $a \leq b$. We will express stack heights on the two \emph{squares} $(p_0,p_2,p_4,p_6)$ and $(p_1,p_3,p_5,p_7)$ in relationship to their respective square minimum, that is,  $\p=(a,b+x_1,a+x_2,b+x_3,a+x_4,b+x_5,a+x_6,b+x_7)$. 

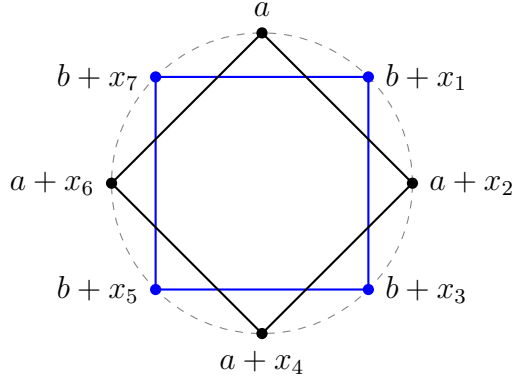
\begin{figure}[!h]
     \centering
     
     \begin{tikzpicture}[scale=2]
\draw[gray, dashed] (0,0) circle (1);
\foreach \i in {0,...,7} {
    \coordinate (P\i) at ({cos(90-45*\i)},{sin(90-45*\i)});
}

\draw[blue, thick] (P1) -- (P3) -- (P5) -- (P7) -- cycle;

\draw[black, thick] (P0) -- (P2) -- (P4) -- (P6) -- cycle;

\node[circle,fill=black,inner sep=1.5pt,label=above:\textbf{$a$}] at (P0) {};
\node[circle,fill=blue,inner sep=1.5pt,label=right:\textbf{$b+x_1$}] at (P1) {};
\node[circle,fill=black,inner sep=1.5pt,label=right:\textbf{$a+x_2$}] at (P2) {};
\node[circle,fill=blue,inner sep=1.5pt,label=right:\textbf{$b+x_3$}] at (P3) {};
\node[circle,fill=black,inner sep=1.5pt,label=below:\textbf{$a+x_4$}] at (P4) {};
\node[circle,fill=blue,inner sep=1.5pt,label=left:\textbf{$b+x_5$}] at (P5) {};
\node[circle,fill=black,inner sep=1.5pt,label=left:\textbf{$a+x_6$}] at (P6) {};
\node[circle,fill=blue,inner sep=1.5pt,label=left:\textbf{$b+x_7$}] at (P7) {};

\end{tikzpicture}
      \caption{Generic labeling of positions of \CN{8}{3}.}
     \label{fig:P(8-3)}
\end{figure}

As a result, the vectors $\z_1=(1,0,1,0,1,0,1,0)$ and $\z_2=(0,1,0,1,0,1,0,1)$ are invariant for any set of positions that are defined solely by a function of the excess $(x_1,x_2,\ldots, x_7)$, because subtraction of $c_i \z_i$ reduces the minimum in each square by the multiple $c_i$ without changing the excess values $x_1,
\ldots,x_7$. 
There are two possible relative locations of the square minima (and subsequently, the zeros of the invariance-reduced position $\tilde{\p}$), either adjacent or distance three apart. Without loss of generality, we assume that  $x_1=0$ for adjacent minima and that $x_3=0$ when there are minima at distance three apart. 

If the minima are adjacent, then $\tilde{\p}=(0,0,x_2,\dots, x_7)$, and zero reduction yields the game \PN{6}{3} with associated reduced position $\bar{\p}=(x_2,\dots, x_7)$. If the minima are at distance three apart, then $\tilde{\p}=(0,x_1,x_2,0,x_4,x_5,x_6, x_7)$ and zero reduction gives $G=\{\{1,2\},\{2,4\},\{4,5,6\},\{5,6,7\},\{7,1\}\}$. Now performing a merge reduction combining vertices $5$ and $6$ results in $\bar{G}=\{\{1,2\},\{2,4\},\{4,x\},\{x,7\},\{7,1\}\}\cong $ \CN{5}{2} with associated reduced position $\bar{\p}=(x_1,x_2,x_4,x_5+x_6,x_7)$. 
Since positions $\p$ and their associated reduced positions $\bar{\p}$ belong to the same class, we define two sets that will be relevant for the description of the \P-positions of \CN{8}{3}. For $\p \in $ \CN{8}{3}, let 
$$S_1=\{\p \mid x_1=0  \text{ and } \bar{\p}=(x_2,x_3,x_4,x_5,x_6,x_7)\in \P_{\PN{6}{3}}\}$$
and 
$$S_2=\{\p \mid x_3=0  \text{ and } \bar{\p}=(x_1,x_2,x_4,x_5+x_6,x_7)\in \P_{\CN{5}{2}}\}.$$
We first show that $S_1 \subset S_2$.

\begin{lemma} \label{lem:S1inS2} Let $\p \in $ \CN{8}{3}. Then $S_1 \subset S_2$.
    \end{lemma}

\begin{proof} Since $\p \in S_1$, we have that $\bar{\p}=(x_2,x_3,x_4,x_5,x_6,x_7)\in \P_{\PN{6}{3}}$. By Theorem~\ref{thm:P-pos path}, $\bar{\p}$ contains a pair of consecutive zeros within $(x_3,x_4,x_5,x_6)$ and any such pair results in either $x_3=0$ or $x_5=0$. This implies  that $\p$ also has minima at distance three apart, so $\bar{\p}=(x_1,x_2,x_4,x_5+x_6,x_7)$ belongs to sub-game \CN{5}{2}. Since $\p \in S_1$, we know that it is a \P-position of \PN{6}{3}, and by Remark~\ref{rem:zero red}, $\p$ is a \P-position of \CN{8}{3}. Applying Remark~\ref{rem:zero red} and Lemma~\ref{lem:merge red} to $\p$ we obtain that  $\bar{\p}=(x_1,x_2,x_4,x_5+x_6,x_7)\in \P_{\CN{5}{2}}$, so $\p \in S_2$.    
\end{proof}

Before we state our result on the \P-positions of \CN{8}{3} in Theorem~\ref{thm:P83},  we need one additional lemma that will simplify the proof of that main result. Let $\p$ be a position of \CN{8}{3} with minima at distance three apart and associated reduced position $\bar{\p}=(y_1,y_2,y_3,y_4,y_5):=(x_1,x_2, x_4, x_5+x_6,x_7)$.  For a legal move $\p \rightarrow \p'$ with associated reduced position $\bar{\p}'=(y_1', y_2', y_3', y_4', y_5')$, we define the following two effects:
\begin{itemize}  
    \item {\bf Type A}: $(+d,+d,+d',\,\leq,\,\leq)$ where $0< d' \leq d$
    \item {\bf Type B}: $(+d,+d,\,\leq,\,\leq,\,\leq)$   
\end{itemize}
The effect describes how $y_i'$ compares to $y_i$: $+d$ indicates that $y_i'=y_i+d$ (and likewise for $d'$), while $\leq$ indicates that $y_i' \leq y_i$.  That is, under an effect of type A, 
\[
\bar{\p}'  = (y_1 + d, y_2 + d, y_3 + d', y_4', y_5'),
\]
and under an effect of type B, 
\[
\bar{\p}'  = (y_1 + d,\ y_2 + d,\ y_3',\ y_4',\ y_5').
\]
Keep in mind that the values $y_i$ and $y_i'$ are expressed in relationship to the current minimum in each of the squares of positions $\p$ and $\p'$, respectively, so the values $y_i'$ can indeed increase. Furthermore, the effect applies to \underline{all five stacks}, even though the move on $\p$  is on at most three consecutive stacks. We have the following lemma.

\begin{lemma} \label{lem:effects} For a position $\p \in S_2$ with $\bar{\p} = (y_1, y_2, y_3, y_4, y_5) \in \P_{\CN{5}{2}}$, if a move from $\p$ to $\p'$ results in either effect A or effect B for the associated reduced position, then  $\bar{\p}' \notin \P_{\CN{5}{2}}$.
\end{lemma}

\begin{proof}
 Let $M=\max(\bar{\p})$, $M'=\max(\bar{\p}')$, and $\P_{5,2}=\P_{\CN{5}{2}}=\{(M,m,c,d,m) \mid  M+m=c+d\}$. Assume that $\bar{\p} \in \P_{5,2}$ and that the move $\p \rightarrow \p'$  has either effect A or B on $\bar{\p}'$. To show that $\bar{\p}' \notin \P_{5,2}$, we consider each possible location of the maximal stack $M$ of $\bar{\p}$ and show in each sub-case that one of the conditions of $\P_{5,2}$ is not satisfied for $\bar{\p}'$. Typically, the sum condition $M+m=c+d$ is not satisfied. 
 \begin{itemize}
     \item $y_1=M$: Under either effect, $y_1'=M'$. Then 
    $y_1' + y_2' = y_1 + y_2 + 2d > y_3 + y_4 +d \geq y_3' + y_4'$, so the sum condition does not hold.

    \item $y_2=M$: Under either effect, $y_2'=M'$ and  $ y_1' + y_2' = y_1+y_2+2d =y_4 + y_5 + 2d >y_4' + y_5'$, thus  $\bar{\p}' \notin P_{5,2}$.  
    \item $y_3=M$: Since \( \bar{\p} \in \P_{5,2} \), we have \( y_2 = y_4 = m \) and \( y_1 + y_5 = y_3 + y_4 \). Under either effect,  \( M' \) cannot occur at \( y_2' \) or \( y_4' \), since \( y_2' \leq y_1' \) and \( y_4' < y_1' \). We consider the other possibilities.
\begin{itemize}
     \item[-] If \( M' = y_1' \), then we require $y_1' + y_2' = y_3' + y_4'$, but
    $y_1' + y_2' \geq y_3' + y_2' = y_3' + y_2+d=y_3' + y_4+d>y_3' + y_4'$.
    \item[-]  If \( M' = y_3' \), then we must have \( y_2' = y_4' = m' \), but \( y_2' = y_2 + d > y_4 \geq y_4' \).
    \item[-] If \( M' = y_5' \), then \( y_1' = y_4' = m' \) must hold, but \( y_1' = y_1 + d\geq m+d > y_4 \geq y_4' \).   
   \end{itemize}
   \item $y_4=M$: Since \( \bar{\p} \in \P_{5,2} \), we have \( y_3 = y_5 = m \) and \( y_1 + y_2 = y_4 + y_5= y_4 +y_3\). After either type of effect, $y_5'$ cannot attain the new maximum \( M' \), because $y_5' \leq y_5 \leq y_1 <y_1+d=y_1'$. Also, $y_3'$ cannot be a unique new maximum because $y_3' \leq y_3+d' \leq y_1+d=y_1'$.  We now consider the other possibilities: 
\begin{itemize} 
    \item[-] If \( M' = y_1' \) or \( M' = y_4' \), then $y_1'+y_2'=y_1+y_2+2d > y_3+y_4+d \geq y_3'+y_4'$, so the required equality $y_1' + y_2' = y_3' + y_4'$ does not hold. 
    \item[-] If \( M' = y_2' \), then $y_1' + y_2' = y_1 + y_2 + 2d = y_4 + y_5 + 2d >  y_4'+ y_5'$,
so the needed equality fails.
\end{itemize}
    \item $y_5=M$: Since \( \bar{\p} \in \P_{5,2} \), we have that \( y_1 = y_4 = m \) and \( y_2 + y_3 = y_5 + m \). The effect of type~B is symmetric to the case \( y_3 = M \) (via reflection). Under effect A,  \( y_4' \) cannot attain the new maximum \( M' \) because  \( y_4' < y_4 + d \leq y_2 + d = y_2' \). Also, $y_1'$ cannot be a unique maximum since $y_1' \leq y_2'$.  We consider the remaining possibilities for \( M' \):
\begin{itemize}
 \item[-] If \( M' = y_2' \) or \( M' = y_5' \), then the required condition is
    \(
    y_5' + y_4' = y_2' + y_3'.
    \)
    However, $y_5' + y_4' \leq y_5 + y_4 = y_2 + y_3 < y_2 + y_3 + d + d' = y_2' + y_3'$.
    \item[-] If  \( M' = y_3' \), then we need  $y_4' = y_2'$, but 
$y_4' \leq y_4 \leq y_2<y_2+d=y_2'$.
\end{itemize}
 \end{itemize}
 Therefore, in all cases,  \( \bar{\p}' \notin \P_{5,2} \).
\end{proof}

We are now ready to state our main result.

\begin{theorem} \label{thm:P83} Let $\p$  be a position of \CN{8}{3}. Then $$\P_{8,3}=\{\p \mid x_3=0  \text{ and } \bar{\p}=(x_1,x_2,x_4,x_5+x_6,x_7)\in \P_{\CN{5}{2}}\}.$$    
\end{theorem}

\begin{proof}  By Lemma~\ref{lem:S1inS2}, we only need to consider \P-positions of \CN{8}{3} that have minima at distance three apart, that is, positions represented by $S_2$. First we prove that there is always a move from $\p \notin S_2$ to an option $\p' \in S_2$. For a position in \CN{8}{3} with zeros at distance three apart, the reduced position $\bar{\p}$ is in \CN{5}{2}, and since $\p \notin S_2$, $\bar{\p} \notin \P_{\CN{5}{2}}$. Thus there exists a move $\bar{\bm{m}}$ to  an option $\bar{\p}' \in \P_{5,2}$ and  an associated move $\bm{m}$ from $\p$ to $\p' \in S_2$. 

Next, we prove that there is no move from  $ S_2$ to $ S_2$.  We first consider moves on $\p \in S_2$ in which \(a\) and \(b\) continue to be the minima of their respective squares. Let $\bm{m}$ denote the move from $\p$ to $\p'$. Since the location of the minima does not change, $\p$ and $\p'$ are in the same subspace after the reduction, namely in game \CN{5}{2}, and there is a unique corresponding move $\bar{\bm{m}}$  from $\bar{\p}$ to $\bar{\p}'$. Since $\bar{\p} \in \P_{5,2}$, we have $\bar{\p}' \notin \P_{5,2}$, and consequently, $\p' \notin S_2$. 

There are also no moves that would create two new minima at distance three apart, since we can only play on three consecutive stacks in \CN{8}{3}. That leaves moves where one new minimal stack is created.  Due to rotational symmetry in the game play of \CN{8}{3}, we only need to consider the following four representative types of move sets on $\p$:  
\begin{enumerate}
    \item \(\{b + x_1, a + x_2, b\}\) (similarly for \(\{a, b + x_1, a + x_2\})\), where play is on one of the minimal stacks and the stacks between the two minimal stacks. 
    \item \(\{a + x_2, b, a + x_4\}\) (similarly for \(\{b + x_7, a,  b + x_1\})\),  where play is one minimal stack and the two adjacent stacks. 
     \item \(\{b, a + x_4, b + x_5\}\) (similarly for \(\{a + x_6, b + x_7, a\})\) where play is on one of the minimal stacks and on two subsequent stacks, none of which is adjacent to the other minimal stack.
    \item \(\{a + x_4, b + x_5, a + x_6\}\) (similarly  for \(\{b + x_5, a + x_6, b+x_7\})\), where play is on neither minimal stack.   
\end{enumerate}
Note that we do not distinguish which is the smaller minimum value in this analysis, and assume, without loss of generality, that the move creates a new minimum value $b'$. Using the fact that any position with adjacent minima also has minima at distance three, we only need to consider those moves that do not create adjacent minima without minima at distance three also. We define $d=b-b'$ and express the option $\p'$ in terms of the new minimum $b'$,  replacing $b$ by $b'+d$ for unplayed stacks. Recall that   
$\p = (a,\, b + x_1,\, a + x_2,\, b,\, a + x_4,\, b + x_5,\, a + x_6,\, b + x_7)$ with  associated reduced position $\bar{\p} = (x_1,\, x_2,\,  x_4, \, x_5 + x_6, \, x_7) \in \P_{5,2}.$
\begin{itemize}
\item[1.] The only moves to consider are of the form \(\{b + x_1,\, a + x_2,\, b\} \rightarrow \{b' + x_1',\, a + x_2',\, b'\}\), where $b' < b$, $0< x_1' \leq x_1+d$, and $0< x_2' \leq x_2$. The resulting position is
$$
\p' = (a,\, b' + x_1',\, a + x_2',\, b',\, a + x_4,\, b' + d + x_5,\, a + x_6,\, b' + d + x_7),
$$
with an associated reduced position 
$
\bar{\p}' = (x_1',\, x_2',\, x_4,\, x_5 + x_6 + d,\, x_7 + d),
$
which (after rotation) displays either a type B effect $(+d,+d,\leq,\leq,\leq)$ if $0<x_1'\leq x_1$ or a type A effect $(+d,+d,+d',\leq,\leq)$ if $x_1<x_1'\leq x_1+d$. In either case, Lemma~\ref{lem:effects} implies that $\bar{\p}' \notin \P_{5,2}$.

\item[2.] Here the only moves to be considered are  
$\{a + x_2, b, a + x_4\} \rightarrow \{a + x_2',\, b',\, a + x_4'\}$
with $b'<b$, $0< x_2' \leq x_2$, and $0< x_4' \leq x_4$. The move leads to  
$$\p' = (a,\, b' + d + x_1,\, a + x_2',\, b',\, a + x_4',\, b' + d + x_5,\, a + x_6,\, b' + d + x_7)$$
with  $\bar{\p}' = (x_1+d,\, x_2',\, x_4',\, x_5 + x_6+d,\, x_7+d).$ Since $\z=\bm{1}$ is  invariant for \(\P_{5,2}\) (see Example~\ref{ex:Ppos52}), we have that
\[
\bar{\p} + d \cdot \z = (x_1 + d,\, x_2 + d,\, x_4 + d,\, x_5 + x_6 + d,\, x_7 + d)
\]  
also belongs to \(\P_{5,2}\). It is easy to verify that there is a legal move  $\bar{\p} + d \cdot \z \rightarrow \bar{\p}'$ in the game play of \CN{5}{2}, thus  \(\bar{\p}' \notin \P_{5,2}\), which implies that  $\p' \notin S_2$.

\item[3.] 
Here we need to consider two types of moves, depending on whether the location of the $b$ minimum remains the same or not. In the first case, the moves are of the form  $\{b,\, a + x_4,\, b + x_5\} \rightarrow \{b',\, a + x_4',\, b' + x_5'\}$, where $b' < b$, $0< x_4' \leq x_4$, and $0< x_5' \leq x_5+d$. 
The resulting position is
$$
\p' = (a,\, b' + d + x_1,\, a + x_2,\, b',\, a + x_4',\, b' + x_5',\, a + x_6,\, b' + d + x_7),
$$
with an associated reduced position 
$
\bar{\p}' = (x_1 + d,\, x_2,\, x_4',\, x_5' + x_6,\, x_7 + d), 
$
which (after rotation and reflection) displays either a type B effect (if $0<x_5'\leq x_5$) or a type A effect (if $x_5<x_5'\leq x_5+d$).

In the second case, when the minimum changes location, then the moves are of the form $\{b,\, a + x_4,\, b + x_5\} \rightarrow \{b'+d',\, a + x_4',\, b' \}$, where  $0< d' \leq d$, $0< x_4' \leq x_4$, and $b' < b$. 
The resulting position is
$$
\p' = (a,\, b' + d + x_1,\, a + x_2,\, b'+d',\, a + x_4',\, b',\, a + x_6,\, b' + d + x_7),
$$
with an associated reduced position 
$
\bar{\p}' = (x_1 + d,\, x_2+d',\, x_4',\,  x_6,\, x_7 + d),
$
because the minima are now at the $p_0'$ and $p_5'$ locations, so stacks 2 and 3 get merged. This represents a Type A effect (after rotation).  By Lemma~\ref{lem:effects}, in either case, we have $\bar{\p}' \notin \P_{5,2}$.
\item[4.] For this case, we only need to consider moves of the form  $
\{a + x_4, b + x_5, a + x_6\} \rightarrow \{a + x_4',\, b',\, a + x_6'\}
$ 
with $b' < b$, $0< x_4' \leq x_4$, and $0< x_6' \leq x_6$ resulting in  
$$
\p' = (a,\, b' + d + x_1,\, a + x_2,\, b' + d,\, a + x_4',\, b',\, a + x_6',\, b' + d + x_7).
$$
with an associated reduced position 
$
\bar{\p}' = (x_1 + d,\, x_2 + d,\, x_4',\,  x_6',\,x_7 + d).
$
Note that the location of the new minimum $b'$ differs from the location of the minimum $b$, so different stacks get combined in the reduced position. The remainder of the argument follows as in case 2 moves. 
\end{itemize}
This completes the proof that there are no moves from $S_2$ to $S_2$, so $S_2 $ represents the \P-positions of the game \CN{8}{3}.
\end{proof}

Before we compare our results to those for simplicial complexes and discuss open questions, we want to remark that Lemma~\ref{lem:S1inS2} is a special case of a broader result, namely that when two sub-games overlap, then the positions in the overlap either have to be both \P-positions or both \N-positions of the overall game, using the same argument as  in Lemma~\ref{lem:S1inS2}. This fact is particularly useful in cases where one subspace is contained in another, because it cuts down on the cases to be considered.

\section{Detour into Simplicial Complexes} \label{sec:Simp}

 The key ingredient in our proofs was the use of invariant vectors. There is some overlap between invariant vectors and the circuits of simplicial simplexes, and some of our results on the structure of the \P-positions resemble those for certain types of circuits. We will use this section to point out the similarities and the differences. One main advantage of invariant vectors is that they can be used for \ruleset{CircularNim} games with larger values of $n$ and $k$, for which many of the results of \cite{ES1996} no longer apply. We will start with some definitions and results and provide examples where invariant vectors shine. 

\begin{definition} \label{def:simpcomp}
 A simplicial complex $\Delta$ on a finite set $V$ is a collection of subsets of $V$ such that
 \begin{enumerate}
     \item $\{v\} \in \Delta$ for all $v\in V$.
     \item If $F \in \Delta$ and $G \subseteq F$, then $G \in \Delta$.
 \end{enumerate}
The members of $\Delta$ are called \emph{simplices} or \emph{faces}, and the elements of $V$ are called \emph{vertices}. A maximal face with respect to inclusion is called a \emph{facet}. A set $C \subseteq V$ is called a \emph{circuit of $\Delta$} if  $C$ is not a face of $\Delta$, but every proper subset of $C$ is a face, that is, $C$ is a minimal non-face. 
\end{definition}

 In the context of simplicial complexes, the game \ruleset{SimplicialNim} is defined as follows.
 
\begin{definition}
    The game \ruleset{SimplicialNim} SimNim$(\Delta)$ is played on a simplicial complex $\Delta$. On each vertex of the complex, place a stack of tokens. In each move, a player removes any (or all) tokens from each of the stacks that form a face of the complex, but must remove at least one token from one of the stacks.  The last player to move wins. 
\end{definition}

Note that this game is identical to \SN{n}{\A}, where $\A$ is the set of facets of $\Delta$. Also, in the context of \ruleset{SetNim}, a circuit $C$ of a simplicial complex is just a  minimal non-playable set. Taking away any one of the vertices of $C$ produces a valid move set of \SN{n}{\A}. Now let $\bm{e}(v)$ denote the $v^{\text{th}}$ unit vector and let  $\bm{e}(A)=\sum_{v\in A} \bm{e}(v)$ for $A \subseteq V$. Then there is the following result. 

\begin{lemma} [Lemma 2~\cite{ES1996}]
If $C$ is a circuit of the simplicial complex, then $n\cdot\bm{e}(C)$ is a zero position\footnote{\P-position} for all $n\in \mathbb{N}$.
    \end{lemma}
    
We do have a corresponding result for invariant vectors, namely that every linear combination of invariant vectors of $\P$ are \P-positions, which follows immediately from the definition of invariance. So what is the difference between invariant vectors and circuits? The following two examples show the difference.  While for some games the set of circuits and the set of invariant vectors can be the same, they typically are not. 

\begin{example} \label{ex:circ=inv} In \CN{6}{3}, the invariant vectors are $\z_1=(1,0,0,1,0,0)$, $\z_2=(0,1,0,0,1,0)$, $\z_3=(0,0,1,0,0,1)$, $\z_4=(1,0,1,0,1,0)$, and $\z_5=(0,1,0,1,0,1)$. Taking a vertex away from the set indexed by $\z_1, \z_2$, and $\z_3$, respectively leaves a single vertex, which is always playable. Taking a vertex away from the sets indexed by either $\z_4$ or $\z_5$ leaves two vertices at distance three, which are playable in \CN{6}{3}. Thus the sets indexed by the invariant vectors are circuits, and in this case, these are exactly the circuits of the complex, so the two sets agree.  
\end{example}

\begin{example} \label{ex:circ_n=_inv} In \CN{7}{3}, the only invariant vector is $\z=\bm{1}$. The set indexed by $\z$ cannot be a circuit, because removing a vertex would allow play on six stacks, so does not represent a move set.  
\end{example}

So far we have seen that the \P-positions contain  the set of linear combinations resulting from either circuits or invariant vectors. However, in the context of simplicial complexes, there is stronger result on the structure of \P-positions based on pointed circuits. 

\begin{definition} A circuit $C$ of $\Delta$ is \emph{pointed} if it has a vertex that does not belong to any other circuit of $\Delta$. Such a vertex is called a \emph{point of} $C$, and $C$ is called \emph{pointed by $v$}. If every circuit of $\Delta$ is pointed, then $\Delta$ is called a \emph{pointed circuit complex}. 
\end{definition}

\begin{theorem} [Theorem 3.2~\cite{ES1996}] \label{thm:pointcirc}
Let $\Delta$ be a pointed circuit complex and let $\mathcal{C}$ be the collection of circuits of $\Delta$. Then the zero positions of $\Delta$ are those of the form 
$$\bm{m}=\sum_{C\in \mathcal{C}}a_C\cdot \bm{e}(C),$$
where $a_C$ is a non-negative integer for each $C$.
\end{theorem}

This theorem is applied in~\cite{ES1996} to obtain the set of \P-positions of a game that is \ruleset{PathNim} \PN{5}{3} with one one additional move, namely the move on the two end stacks. 

\begin{example} [Example 3.3~\cite{ES1996}] \label{ex:33ES}
    The complex $\Delta$ with facets 
    $\A=\{\{1,2,3\},\allowbreak \{2,3,4\},\allowbreak \{3,4,5\},  \{1,5\}\}$
    has circuits $\{1,3,5\},\{1,4\}$, and $\{2,5\}$. These are pointed circuits with points $3$, $4$, and $2$, respectively, and therefore, the set of \P-positions of the game \SN{5}{\A} are given by 
$$a\cdot \bm{e}(\{1,3,5\})+b\cdot \bm{e}(\{1,4\})+ c\cdot \bm{e}(\{2,5\})=(a+b,c,a,b,a+c).$$
\end{example} 

Now consider the game $H$ of Section~\ref{sec:pathnim}, which has the same structure as the complex of Example~\ref{ex:33ES}, but with one additional vertex. 

\begin{example} \label{ex:Hgame} The complex $\Delta$ with facets $\A=\{\{1,2,3\},\allowbreak\{2,3,4\},\allowbreak\{3,4,5\}, \allowbreak\{4,5,6\},\allowbreak\{1,6\}\}$ has  circuits 
$\mathcal{C}=\{\{1,4\},\{1,5\},\{2,5\},\{2,6\},\{3,6\}\}$. Of these, only $\{1,4\}$ and $\{3,6\}$ are pointed, so $\mathcal{C}$ is not a pointed circuit complex.
    \end{example}

Thus, for the game $H$, we cannot use Theorem~\ref{thm:pointcirc} to obtain the \P-positions.  However, we were able to solve this game by using invariant vectors.

\section{Conclusion and Future Work} 
\label{sec:future} 

We have provided a process that assists in not only finding patterns of \P-positions, but also considerably simplifies the proof that there is a move from any position $\p \notin \P$ to an option $\p' \in \P$. Specifically, we have proved results on families of \ruleset{PathNim} games where play is on at least half the stacks, the game $H = $ \SN{6}{\A} with $\A=\{\{a,b,c\},\{b,c,d\},\{c,d,e\},\{d,e,f\},\allowbreak\{a,f\}\}$, and the \ruleset{CircularNim} games \CN{7}{3} and \CN{8}{3}. Invariant vectors played an important role, and we have seen in  Section~\ref{sec:Simp} that they differ from circuits of \ruleset{SimplicialNim} games. In addition, we have found games that can be solved using the \IRP, but not with the structural results on circuits. We conclude that invariance is a more widely applicable property and point out that Proposition~\ref{prop:invPpos} is a companion result  to Theorem 3.2 of~\cite{ES1996}, giving conditions under which the set of \P-positions equals the set of linear combinations of the invariant vectors. We are interested in finding additional structural results  that classify the \P-positions in terms of invariant vectors. 

\begin{open} Under what conditions are the \P-positions of a game \SN{n}{\A} given by the linear combinations of the invariant vectors?
\end{open}

Additional open questions concern the existence or non-existence of invariant vectors. We have seen that some structures in the \P-positions limit the number and form of the invariant vectors. In several games, the vector $\bm{1}$ was the only invariant vector of the game. In particular, it is the only invariant vector for the family of games  \CN{n}{n-1}, whose \P-positions are exactly those that have equal stack heights. 

\begin{open} Under what conditions is the invariant vector $\bm{1}$  the only invariant vector of a game?
\end{open}

By Remark~\ref{rem:inv-spec}(6), we know that invariant vectors can only exist for $k>n/3$. Solved \ruleset{CircularNim} games, unpublished results on \CN{9}{4}, \CN{10}{4}, and  \CN{11}{4},  as well as additional computational evidence suggest that groups of games for which the number of allowed moves $k$ is approximately $n/3$  have invariant vector $\bm{1}$ (and possibly others).  We suspect that with the exception of \CN{n}{n-1}, these are the only games to have invariant vectors of their set of \P-positions. 

\begin{conjecture} The only \ruleset{CircularNim} games that admit invariant vectors  are the games  \CN{3k}{k+1}, \CN{3k+1}{k+1}, \CN{3k+2}{k+1} with $k \geq 1$ and \CN{n}{n-1}.  More specifically,  $\bm{1}$ is one of the invariant vectors.   
\end{conjecture}

The Matryoshka effect of the \IRP, where we not only find smaller sub-games nestled within  larger games, but also observe that in all instances the sub-games themselves have invariant vectors, leads to the following conjecture.  

\begin{conjecture} Let $G$ be a \ruleset{SetNim} game that admits invariant vectors. Then every (non-trivial)\footnote{Trivial here means a single \Nim~stack.}  sub-game of $G$ that arises through the \IRP~also admits invariant vectors. 
\end{conjecture}

As mentioned earlier, to solve \CN{6}{2}, we first need to solve the \ruleset{SetNim} game $G\defeq \{\{a,b\}, \{b,c\}, \{d\}\}$. By Example~\ref{ex:red_fam}(4),  as a sub-game of \CN{6}{2}, the same game is also a sub-game of  \CN{3k}{k}. Solving it is a crucial first step in solving all these other games. 

\begin{open} Solve the game \(\SN{4}{\{\{a,b\}, \{b,c\}, \{d\}\}}.\) 
\end{open}

In summary, the notion of invariant vectors provides a rich field of exploration, and an avenue to solve  families of \ruleset{CircularNim} and other \ruleset{SetNim} games.   We look forward to progress in this area.

\section*{Acknowledgements}
Balaji R. Kadam gratefully acknowledges the support and valuable discussions provided by A. J. Shaiju during the course of this research. 

\printbibliography

@book {Sie2013,
    AUTHOR = {Siegel, Aaron N.},
     TITLE = {Combinatorial game theory},
    SERIES = {Graduate Studies in Mathematics},
    VOLUME = {146},
 PUBLISHER = {American Mathematical Society, Providence, RI},
      YEAR = {2013},
       DOI = {10.1090/gsm/146},
       URL = {https://doi.org/10.1090/gsm/146},
}

@book {AlNoWo2019,
    AUTHOR = {Albert, Michael H. and Nowakowski, Richard J. and Wolfe, David},
     TITLE = {Lessons in Play},
PUBLISHER = {CRC Press, Boca Raton, FL},
      YEAR = {2019},
}

@book {BeCoGu2001,
    AUTHOR = {Berlekamp, Elwyn R. and Conway, John H. and Guy, Richard K.},
     TITLE = {Winning ways for your mathematical plays. {V}ol. 1-4},
   EDITION = {Second Ed.},
 PUBLISHER = {A K Peters, Ltd., Natick, MA},
      YEAR = {2001-2004},
}

@article {DHV21,
    AUTHOR = {Dufour, Matthieu and Heubach, Silvia and Vo, Anh},
     TITLE = {Circular {N}im games {CN}(7,4)},
   JOURNAL = {Integers},
  FJOURNAL = {Integers. Electronic Journal of Combinatorial Number Theory},
    VOLUME = {21B},
      YEAR = {2021},
     PAGES = {Paper No. A9, 18},
}

@article {DuHe09,
    AUTHOR = {Dufour, Matthieu and Heubach, Silvia},
     TITLE = {Circular {N}im games},
   JOURNAL = {Electron. J. Combin.},
  FJOURNAL = {Electronic Journal of Combinatorics},
    VOLUME = {20},
      YEAR = {2013},
    NUMBER = {2},
     PAGES = {Paper 22, 26},
       DOI = {10.37236/2765},
       URL = {https://doi.org/10.37236/2765},
}

@article {GHHC20,
    AUTHOR = {Gurvich, Vladimir and Heubach, Silvia and Ho, Nhan Bao and
              Chikin, Nikolay},
     TITLE = {Slow {$k$}-{N}im},
   JOURNAL = {Integers},
  FJOURNAL = {Integers. Electronic Journal of Combinatorial Number Theory},
    VOLUME = {20},
      YEAR = {2020},
     PAGES = {Paper No. G3, 19},
}

@article {KaSh2025,
    AUTHOR = {Kadam, Balaji R. and Shaiju, A. J.},
     TITLE = {Variations of the {N}im game played on a cube},
   JOURNAL = {Integers},
  FJOURNAL = {Integers. Electronic Journal of Combinatorial Number Theory},
    VOLUME = {25},
      YEAR = {2025},
     PAGES = {Paper No. G2, 14},
}

@article {Hor2010,
    AUTHOR = {Horrocks, David},
     TITLE = {Winning positions in simplicial {N}im},
   JOURNAL = {Electron. J. Combin.},
  FJOURNAL = {Electronic Journal of Combinatorics},
    VOLUME = {17},
      YEAR = {2010},
    NUMBER = {1},
     PAGES = {Research Paper 84, 13},
       DOI = {10.37236/356},
       URL = {https://doi.org/10.37236/356},
}

@article {ES1996,
    AUTHOR = {Ehrenborg, Richard and Steingr\'imsson, Einar},
     TITLE = {Playing {N}im on a simplicial complex},
   JOURNAL = {Electron. J. Combin.},
  FJOURNAL = {Electronic Journal of Combinatorics},
    VOLUME = {3},
      YEAR = {1996},
    NUMBER = {1},
     PAGES = {Research Paper 9, approx. 33},
       DOI = {10.37236/1233},
       URL = {https://doi.org/10.37236/1233},
}

@article {Hol1958,
    AUTHOR = {Holladay, John C.},
     TITLE = {Matrix {N}im},
   JOURNAL = {Amer. Math. Monthly},
  FJOURNAL = {American Mathematical Monthly},
    VOLUME = {65},
      YEAR = {1958},
     PAGES = {107--109},
       DOI = {10.2307/2308886},
       URL = {https://doi.org/10.2307/2308886},
}

@article {Mo1910,
    AUTHOR = {Moore, E. H.},
     TITLE = {A generalization of the game called nim},
   JOURNAL = {Ann. of Math. (2)},
  FJOURNAL = {Annals of Mathematics. Second Series},
    VOLUME = {11},
      YEAR = {1910},
    NUMBER = {3},
     PAGES = {93--94},
       DOI = {10.2307/1967321},
       URL = {https://doi.org/10.2307/1967321},
}

@article {Bo1901,
    AUTHOR = {Bouton, Charles L.},
     TITLE = {Nim, a game with a complete mathematical theory},
   JOURNAL = {Ann. of Math. (2)},
  FJOURNAL = {Annals of Mathematics. Second Series},
    VOLUME = {3},
      YEAR = {1901/02},
    NUMBER = {1-4},
     PAGES = {35--39},
       DOI = {10.2307/1967631},
       URL = {https://doi.org/10.2307/1967631},
}

\end{document}